\documentclass[a4 paper,10pt]{article}
\usepackage{amsmath}
\usepackage{latexsym}
\usepackage{amsfonts}
\usepackage{amssymb}
\usepackage{amsthm}
\usepackage{amsbsy}
\usepackage{graphicx}
\usepackage{color}
\usepackage{epstopdf}

\newcommand{\clb}{\color{blue}}

\newcommand{\cX}{\mathcal{X}}
\newcommand{\nada}[1]{}
\makeatletter
\def\@maketitle{\newpage
    \null
    \vskip .8truein
    \begin{center}%
     {\bf \@title \par}%
     \vskip 1.5em
     {\small
      \lineskip .5em
      \begin{tabular}[t]{c}\@author
      \end{tabular}\par}%
    \end{center}%
    \par
    \vskip .4truein}
\@addtoreset{equation}{section} 
\@addtoreset{theorem}{section} 
\@addtoreset{lemma}{section} \@addtoreset{proposition}{section}
\@addtoreset{definition}{section} \@addtoreset{corollary}{section}
\@addtoreset{remark}{section}

\let\G=\Gamma  

\let\nn=\nonumber

\newcommand{\re}{\mathbb{R}}

\def\R{\mathbb{R}}
\def\Q{\mathbb{Q}}
\def\dP{\mathbb{P}}
\def\G{\mathbb{G}}
\def\N{\mathbb{N}}
\def\dH{\mathbb{H}}
\def\Z{\mathbb{Z}}
\def\cF{\mathcal{F}}
\def\cH{\mathcal{H}}
\def\cX{\mathcal{X}}
\def\dE{\mathrm{e}}
\def\cA{\mathcal{A}}
\def\cB{\mathcal{B}}

\def\Del{\delta_{1/\varepsilon}}

\let\e=\varepsilon

\newtheorem{theorem}{Theorem}[section]
\newtheorem{lemma}{Lemma}[section]
\newtheorem{proposition}{Proposition}[section]
\newtheorem{definition}{Definition}[section]
\newtheorem{defi}{Definition}[section]
\newtheorem{corollary}{Corollary}[section]
\newtheorem{remark}{Remark}[section]
\newtheorem{rem}{Remark}[section]
\newtheorem{example}{Example}[section]
\newtheorem{ex}{Example}[section]

\newcommand{\norma}[1]{\left \|#1\right \|}

\def\proof{\list{}{\setlength{\leftmargin}{0pt}
                      \parskip=0pt\parsep=0pt\listparindent=2em
                      \itemindent=0pt}\item[]\futurelet\testchar\@maybe}

\def\@maybe{\ifx[\testchar \let\next\@Opt
          \else \let\next\@NoOpt \fi \next}
\def\@Opt[#1]{{\it Proof of #1.\ }}\def\@NoOpt{{\it Proof.\ }}

\begin{document}

\title{\Large \bf Stochastic homogenization for functionals with anisotropic rescaling and non-coercive Hamilton-Jacobi equations.
}
\author{Nicolas Dirr\thanks{Cardiff School of Mathematics, Cardiff University, Cardiff, UK, e-mail: dirrnp@cardiff.ac.uk.} \and
Federica Dragoni
\thanks{Cardiff School of Mathematics, Cardiff University, Cardiff, UK, e-mail: DragoniF@cardiff.ac.uk.} \and
Paola Mannucci 
\thanks{Dipartimento di Matematica ``Tullio Levi-Civita'' ,  Universit\`a di Padova, Padova, Italy, e-mail: mannucci@math.unipd.it} \and
Claudio Marchi
\thanks{Dipartimento di Ingegneria dell'Informazione,  Universit\`a di Padova, Padova, Italy, e-mail: claudio.marchi@unipd.it}
}
\maketitle
\begin{abstract}
\noindent We study the stochastic homogenization for a Cauchy problem for a first-order Hamilton-Jacobi equation whose operator is not coercive w.r.t. the gradient variable.
We look at Hamiltonians like $H(x,\sigma(x)p,\omega)$ 
where 
$\sigma(x)$ is a matrix associated to  a Carnot group.
The rescaling considered  is consistent with the underlying Carnot group structure, thus anisotropic.
 We will prove that under suitable  assumptions for the Hamiltonian, the solutions of the $\varepsilon$-problem converge 
 to a deterministic function which can be characterized as the unique (viscosity) solution of a suitable deterministic Hamilton-Jacobi problem.
\end{abstract}
%
%
%
%
%


\section{Introduction} 

Homogenization problems have been studied for many years for both their intrinsic mathematical interest and the many applications in different sciences (e.g. the study of heterogeneous media). In particular stochastic homogenization arises whenever
at the microscopic level the system depends on some random variable but at the macroscopic level one can expect a deterministic behaviour.\\
 In this paper, we study asymptotics of a special class of {\em degenerate} (i.e. non-coercive) first-order Hamilton-Jacobi equations with random coefficients taking the form
\begin{equation}
\label{equation1}
\left\{
\begin{aligned}
&
u_t+H\left(x, \sigma(x)\nabla u,\omega\right)=0,\ t>0,\ x\in\re^N,\  \omega \in \Omega,\\
& u(0,x,\omega)=g(x),\ x\in \re^N, \;\omega \in \Omega,
\end{aligned}
\right.
\end{equation}
where $(\Omega, {\mathcal F},{\mathbb P})$ is a given probability space, and  $\sigma:\R^N\to \R^{m\times N}$  with $m\le N$.  Even though  $H(x,q)$ is coercive and convex in the variable $q=\sigma(x)p\in \R^m,$ the map
$p\mapsto H\left(x, \sigma(x)p,\omega\right)$ is in general not coercive because $\sigma(x)$ may have a nontrivial kernel.
 The illustrating example is the Heisenberg group, which is topologically $\R^3$ but with a  different algebraic structure (see Section~$2$, e.g. \eqref{MatrixHeisenberg}).\\
Equations as \eqref{equation1}  can be understood in the framework of Carnot groups, i.e.  non-commutative stratified nilpotent Lie groups (see Section 2 for more details). 
In particular these groups satisfy the H\"ormander condition:
they are endowed with a family of vector fields that, together with all their associated commutators, 
span the whole tangent space at any point of the original manifold.

For the associated homogenization problem, the Carnot group structure suggests a natural anisotropic rescaling of $\R^N,$ denoted by $\delta_{1/\varepsilon}(x)$ for $x\in \R^N.$ \\
Then the homogenization problem can be formulated as follows:\\
under some assumptions made precise later (see Section 3), find the equation solved by the (locally uniform) limit of $u^\varepsilon(t,x,\omega)$ where $u^{\varepsilon}$ are viscosity solutions of
\begin{equation}
\label{equation2}
\left\{
\begin{aligned}
&
u^{\varepsilon}_t+H\left(\delta_{1/\varepsilon}(x), \sigma(x)\nabla u^{\varepsilon},\omega\right)=0,\ t>0,\ x\in\re^N,\  \omega \in \Omega,\\
& u^{\varepsilon}(0,x,\omega)=g(x),\ x\in \re^N, \;\omega \in \Omega.
\end{aligned}
\right.
\end{equation}
In other words, the aim is to identify $\overline H:\R^m\to \R$ such that the  viscosity solutions of \eqref{equation2} converge, locally uniformly in $t$ and $x$ and almost surely in $\omega$, to a deterministic  function $u(t,x)$ which can be characterized as the unique viscosity solution of a problem of the form 
 \begin{equation}
 \label{LimitProblem}
 \left\{
 \begin{aligned}
 &u_t+\overline{H}\left(\sigma(x)\nabla u\right)=0,\ t>0,\ x\in\re^N,\\
 &u(0,x)=g(x),\ x\in\R^N.
 \end{aligned}
 \right.
 \end{equation}
In the case of the Heisenberg group, the anisotropic rescaling is
 $\delta_{1/\varepsilon}(x_1,x_2,x_3)=(\varepsilon^{-1}x_1,\varepsilon^{-1}x_2,\varepsilon^{-2}x_3).$ This is consistent with the geometric structure of the Heisenberg group, but the anisotropy can be understood heuristically in another way: at each point, some directions are ``forbidden'', i.e. paths of the associated control problem can move only on a two-dimensional subspace. By varying their direction often (i.e. by the use of non-trivial commutators from the H\"ormander condition) they are able to reach any given point but the cost for ``zig-zagging'' to get in the forbidden direction is higher, so typically they move slower in these directions, which makes  a faster rescaling necessary.\\
Note that, in~\eqref{equation2}, $\sigma(x)$ is not rescaled so this is in principle  a problem with a  fast and a slow variable, but the equation is degenerate if the   slow variables are frozen. Obviously, general non-coercive equations have no homogenization, so considering a cell problem with a  frozen variable is not the way to tackle this problem.\\
Instead, our approach  is based  on the use of a variational formulation for the viscosity solutions of \eqref{equation2}, that has been introduced in the coercive case by  Souganidis \cite{S1} and Rezakhanlou-Tarver \cite{RT}.  This variational approach is motivated by $\Gamma$-convergence methods for the random Lagrangian. In order to define the associated variational problem from the Hamilton-Jacobi equation, some form of convexity is needed, but it should be noted that due to the degeneracy the relation is more subtle than the Euclidean Legendre transform, see \cite{BCP}.
Moreover the approach developed  in \cite{RT, S1}  fails  since the idea of using the Subadditive Ergodic Theorem indirectly requires the existence of curves invariant under translation and rescaling (as straight lines are w.r.t. the Euclidean translations). In our anisotropic geometries, this property is true only for curves that have constant horizontal speed (i.e. velocity constant w.r.t. a given family of left-invariant vector fields, see Section 2 for more details). Unfortunately those lines are too few to cover the whole space (they only generate a $m$-dimensional submanifold in $\R^N$ where usually $N>m$).\\
Then main idea of the  proof for the convergence theorem is to apply  the techniques from \cite{RT, S1} (for the periodic case see also \cite{E}) to a lower dimensional constrained variational problem (Section 5), the constraint being to belong to the $m$-dimensional manifold mentioned above. Then by an approximating argument (Section 6) we write the original variational problem  \eqref{Lepsilon}
as limit of sum of lower dimensional constrained variational problems. The key role in the whole argument will be to approximate any horizontal curve by a suitable family of piecewise horizontal lines with constant speed and the use of the H\"ormander condition to move everywhere in the space.\\
 Here our a priori bounds on the Lagrangian ensure that the cost of connecting any two points can be bounded by a function of the geodesic distance. This allows to estimate the difference in cost for connecting nearby points,  
 a property which makes up for the lack of uniform continuity of the Lagrangian due to the rescaling in space. 
\\

This is to our knowledge the first paper which connects two previously separate branches of homogenization theory: Stochastic homogenization on the one hand, which so far has not been considered in sub-Riemannian geometries, and homogenization in the sub-elliptic setting, which so far has been restricted to a suitable generalization of periodic environments, i.e. essentially in a compact setting.
For homogenization in subelliptic settings in the periodic case see for example \cite{BW, BMT, Franchi1, Franchi2, MS, STR1}, and for homogenization with singular perturbation see  \cite{AB1,ABM}.


Since the first results on stochastic homogenization for first-order Hamilton-Jacobi equations (\cite{S1,RT}), it has been a difficult question which are the necessary conditions on the deterministic structure  of the Hamiltonian, with convexity and coercivity being sufficient.  The case of non-convexity has been understood better recently, see e.g. \cite{Armstrong-Carda, FS, Ziliotto}. Instead this paper gives a very general  class of examples which are convex (but not strictly) but non-coercive. This homogenization result is in line with the folk theorem that, in order to have homogenization, characteristics have to be able to go everywhere: our degeneracy is related to H\"ormander vector fields, which have the property that admissible paths (see Section 2) can connect any two given points.\\
$\Gamma$-convergence for random functionals, which is used here, has in a general setting first been studied by Dal Maso and Modica, \cite{DalMasoModica} and recently been extended to non-convex integrands, \cite {DG}.
Alternatives to the variational approach for obtaining stochastic homogenization results in the Euclidean setting for both first and second order equations and the simultaneous effect of homogenization and vanishing viscosity (i.e. singular perturbation) have been developed subsequently, see for example \cite{KRV, LS05, LS10}. Extending these methods to the sub-Riemannian setting will be a challenge for further research.\\

This paper is organized as follows.\\
In Section 2 we introduce some basic notions for Carnot groups, in particular the dilations in the group and some norms and distances related to both the geometric and the algebraic structure of Carnot groups.
In this section we also introduce horizontal curves, horizontal velocity and study some properties, which will be very useful in later proofs.\\
In Section 3 we state the problem and we explain the meaning on some assumptions on the Hamiltonian; in particular the stationary ergodic assumption which is crucial in order to get a deterministic limit problem. 
In the same section we also introduce the variational formulation for the solutions of the $\varepsilon$-problem.
\\
In Section 4 we study several properties for the variational problem. In particular we  prove local uniform continuity.\\
In Section 5 we prove the convergence for the constrained variational problem, i.e. for the minimizing problem for an integral cost under the additional $m$-dimensional constraint.\\
In Section 6 we prove our main convergence result for the unconstrained variational problem by the introduction of a suitable approximation argument.\\
In Section 7 we apply the  convergence proved in Section 6 to the family of non-coercive Cauchy-Hamilton-Jacobi problems \eqref{equation2} via variational formula.\\
In the Appendix (Section 8) we give a proof for the well-posedness of the $\varepsilon$-problem  \eqref{equation2} in the viscosity sense.

\section{Preliminaries: Carnot groups.}
Carnot groups are non-commutative Lie groups: they are endowed both with a non-commutative algebraic structure and with a manifold structure. The lack of commutativity in the algebraic structure reflects on the manifold structure as restrictions on the admissible motions. This means that the allowed curves are constrained to have their velocities in a lower dimensional subspace of the tangent space of the manifold. Then the associated manifold structure is not Riemannian but sub-Riemannian. We refer the reader to \cite{BLU} for an overview on Carnot groups and sub-Riemannian manifolds. Here we only recall the definitions and some of the main properties, which will be crucial in the later proofs.
\begin{defi}\label{defG}[Carnot group]
A Carnot group $(\G, \circ)$ of step $r$ is a simply connected, nilpotent Lie group whose Lie algebra $g$ of left-invariant vector fields admits a stratification, i.e. there exist non zero subspaces $\{V_i\}$, $i=1,\dots r$ such that $g=\bigoplus_{i=1}^r V_i$, $[V_1, V_i]=V_{i+1}\neq 0$, for $i=1,\dots r-1$, $[V_1,V_r]=0$. $V_1$ is called the first layer.\\
Any such group is isomorphic to a homogeneous Carnot group in $\R^N$, that is a triple $(\R^N, \circ, \delta_{\lambda})$ where $\R^N=\R^{n_1}\times\R^{n_2}\times\dots\times\R^{n_k}$, $\circ$ is a group operation whose identity is $\dE$ and such that $(x,y)\rightarrow y^{-1}\circ x$ is smooth (where $y^{-1}$ denote the inverse of $y$), and $\delta_{\lambda}:\R^N\rightarrow \R^N$ is the dilation:
\begin{equation}
\label{Dilations}
\delta_{\lambda}(x)=
\delta_{\lambda} \left(x^{(1)},x^{(2)},\cdots, x^{(r)}\right):=\left(\lambda\,x^{(1)},\lambda^2\,x^{(2)}, \cdots,\lambda^r x^{(r)}\right),\ x^{(i)}\in \R^{n_i},
\end{equation}
is an automorphism of the group $(\R^N, \circ)$ for all $\lambda>0$ and there are $m:=n_1$ smooth vector fields  $X_1$, $\cdots$, $X_m$ on $\R^N$ invariant with respect to the left translation
$$
L_{\beta}(x):=\beta\circ x
$$ 
for all $\beta\in\R^N$ and such that 
they generate a Lie algebra with rank $N$ at every point $x\in\R^N$.
The vector fields $X_1$, $\cdots$, $X_m$  are called the generator of the Carnot groups or horizontal vector fields and the $n\times m$ matrix whose columns are these vector fields is denoted by $\sigma$.

For $x\in\R^N$, we shall also use the notation: $x=(x^1,x^2)$ with $x^1\in\R^m$, $x^2\in \R^{N-m}$ and $x^1:=\pi_m(x)$. 
\end{defi}
The definition of dilations (that replace the role of product of a point by a scalar in  the Euclidean case) gives good notions of rescaling in these geometries.\\

Note that we are interested only in the case where $\G=\R^N$ for some $N\geq 3$ (in fact Carnot groups  with dimension less than 3 do not exist).
\begin{ex}
\label{Heisenberg}
The simplest example of a Carnot group is the so called Heisenberg group.  
The $N$-dimensional Heisenberg group $\dH^N$ is a Carnot group of step 2 (i.e. $r=2$ in the stratification) defined in $\R^{2N+1}$ (with $N\geq 1$).  In particular if $N=1$ the stratification
is $V_1\bigoplus V_2$, where $V_1= \re^2$ and $V_2=\re$.
In this last case
the group operation is
$$
x\circ y:=\left(x_1+y_1, x_2+y_2, x_3+y_3+\frac{x_1y_2- x_2y_1}{2}\right)
$$
where $x=(x_1, x_2, x_3)$ and $y=(y_1, y_2, y_3)$ are two points in $\R^{3}$ and 
the generators are the two vector fields
$$ X_1(x)=\left(\begin{array}{c}1 \\0 \\\frac{x_2}{2}\end{array}\right),\ 
X_2(x)=\left(\begin{array}{c}0 \\1 \\\frac{-x_1}{2}\end{array}\right)
$$
\end{ex}
In the Heisenberg group $\dH^1$ the dilations that give the natural rescaling are
$$
\delta_{\lambda}(x)=\delta_{\lambda} (x_1,x_2,x_3)=(\lambda\,x_1,\lambda\,x_2, \lambda^2\,x_3).
$$
To make the paper more easily readable for mathematicians not used to worked in Carnot groups
we will explain most of the notions and 
properties of Carnot groups, using the 1-dimensional Heisenberg  group $\dH=\dH^1$ as referring model.\\ 

Another family of algebraic objects which will play a crucial role in our homogenization problem are the translations. Since the group law is not commutative, in general left translations and right translations will be different. We will always translate points using only the left translations.\\
Using the stratification, a Carnot group can be endowed with a homogeneous norm that induces a homogeneous distance. The homogeneous norm and the homogeneous distance are very important in homogenization problems since they are compatible with rescaling under dilations (as we will see in the properties  below).
\begin{defi}\label{homogenousNorm}
[Homogeneous norm and homogeneous distance]
A homogeneous norm $\|\cdot\|_h$ is a continuous function from $\G$ to $[0,+\infty)$ such that 
\begin{enumerate}
\item 
$\|x\|_h=0\ \iff\ x=\dE$
\item 
$\|x^{-1}\|_h=\|x\|_h$ 
\item $\|\delta_{\lambda}(x)\|_h=\lambda\|x\|_h, \forall x\in \G, \lambda>0$
\item
$\|x+y\|_h\leq \|x\|_h+\|y\|_h, \forall x,y\in \G.$
\end{enumerate}
The homogeneous distance between two points $x,y\in \G$ is
$$
d_h(x,y)=\|y^{-1}\circ x\|_h.
$$
From $\|x\|_h=\|x^{-1}\|_h$ we have that $d_h(x,y)=d_h(y,x)$ and obviously $d_h(x,x)=0$ for all $x,y\in \G$.\\
Moreover,  given two points $x, y\in \G\equiv \R^N$, $\big(\delta_{\lambda}(x)\big)^{-1}=\delta_{\lambda}(x^{-1})$ and \\
  $\delta_{\lambda}(x)\circ \delta_{\lambda}(y) =\delta_{\lambda}(x\circ y)$.
 This implies that 
 $$d_h\big(\delta_{\lambda}(x),\delta_{\lambda}(y) \big)=\lambda\,d_h(x,y).
 $$
\end{defi}

In the case of the 1-dimensional Heisenberg group $\dH$ we have
$$
\norma{x}_h=\norma{(x_1,x_2,x_3)}_h=\big((x_1^2+x_2^2)^2+x_3^2\big)^{1/4}.
$$
Moreover it is easy to check that $\dE=(0,0,0)$ and $x^{-1}=(-x_1,-x_2,-x_3)$ so
$$
d_h(x,y)=\big(((x_1-y_1)^2+(x_2-y_2)^2)^2+(x_3-y_3)^2\big)^{1/4}.
$$
One can easily check all the properties listed above in the case  of the  1-dimensional Heisenberg group.\\


For later use, it is very useful to  introduce the $m\times n$ matrix associated to the vector fields 
$$
\sigma(x):=\big(X_1(x),\dots,X_m(x)\big)^T
$$
e.g. in $\dH^1$ the matrix $\sigma(\cdot)$ is the $2\times 3$-matrix given by 
\begin{equation}
\label{MatrixHeisenberg}
\sigma(x_1,x_2,x_3)=
\begin{pmatrix}
1 &0& -\frac{x_2}{2}\\
0&1&\frac{x_1}{2}
\end{pmatrix}.
\end{equation}
 
 From now on we will always consider the Carnot groups, written 
 in exponential coordinates (or  canonical coordinates).
 In fact in exponential coordinates the vector fields  (and so the associated matrix $\sigma(x)$) assume a special form, as shown in the following lemma.
 \begin{lemma}
 Given a Carnot group in exponential (or canonical) coordinates, then the vector fields can be considered as the columns of a $m\times N$ matrix $\sigma(x)$ of this form 
 \begin{equation}\label{matrixC}
 \sigma(x)=\begin{pmatrix}Id_{m\times m}&A(x)\end{pmatrix}
 \end{equation}
 where $Id_{m\times m}$ is the identity matrix $m\times m$ and $A(x)$ is a  $m \times (N-m)$  matrix whose coefficients are smooth functions depending only on $x_1,\dots, x_m$.\\
Moreover the non-vanishing coefficients of $A(x)=(a_{j,i}(x))$  with $i=1,\dots,N-m$ and $j=1,\dots,m$ are polynomial functions of degree $k-1$ whenever the $(m+i)$-th component rescale as $\lambda^k$ in the dilations $\delta_{\lambda}$ defined in \eqref{Dilations}.
\end{lemma}  
For a proof we refer the reader to  \cite{BLU}; in particular see  \cite[Proposition 1.3.5, Corollary 1.3.19] {BLU} for the polynomial structure and the corresponding homogeneity degree. Remember that $\delta_{\lambda}$-homogeneity corresponds to Euclidean homogeneity whenever the functions depend only on the first $m$ components.\\

The previous lemma is easy to check in the 1-dimensional Heisenberg group (see \eqref{MatrixHeisenberg}), in fact 
$a_{1,1}(x_1,x_2)=-\frac{x_2}{2} $ and $a_{2,1}(x_1,x_2)=\frac{x_1}{2} $ are both polynomials of degree 2-1=1. We now give another example for a  step 3 Carnot group.
\begin{ex}[Engel group in exponential coordinates]
The Engel group is  Carnot group of step 3 defined on $\R^4$.  It can be written as extension of the Heisenberg group but  for us it is crucial to write it in exponential coordinates (see e.g. \cite{LeDonne}).
The rescaling in the Engel group is given by 
$$
\delta_{\lambda}(x_1,x_2,x_3,x_4)=\big(\lambda x_1,\lambda x_2,\lambda^2 x_3,\lambda^3 x_4
\big).
$$
In exponential coordinate the vector fields  generating $V_1$ can be written as
$$
X_1(x_1,x_2,x_3,x_4)=\frac{\partial}{\partial x_1}\;\textrm{and}
\;X_2(x_1,x_2,x_3,x_4)
=\frac{\partial}{\partial x_2}+x_1\frac{\partial}{\partial x_3}+\frac{x_1^2}{2}\frac{\partial}{\partial x_4}.
$$
In this case the corresponding $2\times 4$-matrix has the form of a $2\times 2$-identity matrix and a $2\times 2$ matrix $A(x)$ whose coefficients are $a_{1,1}(x)=0=a_{1,2}(x)$, while $a_{2,1}(x)=x_1$ which is a polynomial of  degree 1 (in fact the component  $2+1=3$  rescales with $k=2$), and 
$a_{2,2}(x)=\frac{x_1^2}{2}$ which is a polynomial of  degree 2 (in fact the  component $2+2=4$  rescales with $k=3$). Then Lemma \ref{matrixC} is easily verified.
\end{ex}
So far, we have briefly recalled the algebraic structure of Carnot groups.
Since Carnot groups are also sub-Riemannian manifolds there is also another important distance to consider: the so called Carnot-Carath\'eodory distance.
Before defining the Carnot-Carath\'eodory distance and its relations with the homogeneous distance and the Euclidean distance, we need to introduce the sub-Riemannian manifold structure associated to a Carnot group.
 Consider the left-invariant vector fields   $X_1,\dots, X_m$ introduced above on $\R^N$, then by identifying the tangent space at the origin with the Lie algebra $g$ (see Definition \ref{defG}) and in any other point by left-translation, then $X_1,\dots, X_m$  satisfy the H\"ormander condition with step $r$.  We remind that the H\"ormander condition states that the Lie algebra induced by the vector fields has to be at any point equal to the whole tangent space at that point.\\
Denoted by $\cH_x=\textrm{Span}\big(X_1,\dots,X_m\big)$ the distribution spanned by the given left-invariant vector fields, then it is possible to define a  Riemannian metric
on $\cH_x$ induced by the vector fields, by taking $<v,w>=\alpha\cdot\beta$ where $\alpha$ and $\beta$ are $m$-valued vectors, corresponding to the coordinates of $v$ and $w$ respectively, w.r.t. the given vector fields.\\
The triple $\big(\R^N,\cH_x,<\cdot,\cdot>\big)$ is a sub-Riemannian manifold.
For more details on sub-Riemannian manifolds in general and the manifold structure associated to Carnot groups in particular, we refer respectively to \cite{montgomery} and \cite{BLU}. \\
Next we recall the notion of horizontal (or admissible) curve that will play a crucial role in  defining the Carnot-Carath\'eodory distance and 
later in the variational formulas.

\begin{defi}\label{horcurv}
An absolutely continuous  curve $\xi:[0,T]\to \R^N$ is called horizontal if 
 there exists   $\alpha^{\xi}:[0,T ]\to \R^m$ measurable such that
 \begin{equation}
 \label{EQ_Horizontal}
 \dot{\xi}(s)=\sum_{i=1}^{m}\alpha_i^{\xi}(s)X_i(\xi(s)),\quad a.e. \;s\in (0,T),
 \end{equation}
where the vector fields $X_i$ are those introduced in Definition \ref{defG}.\\ 
The vector $\alpha^{\xi}$ is called {\em horizontal velocity} of the curve.
\end{defi}
\begin{rem}
\label{RemarkLinearIndepedent}
Note that whenever $X_1,\dots, X_m$ are linearly independent, as they are  always in the case of Carnot groups (see e.g. \cite[Ch. 1]{BLU}), the vector $\alpha^{\xi}$ is unique up to a measure zero set.
\end{rem}


Let us define the Carnot-Carath\'eodory distance (briefly C-C distance) associated to a family of vector fields $\cX=\{X_1,\dots,X_m\}$.
\begin{defi}
\label{CC-distance_definiton}
Given two points $x,y\in \R^N$ and a family of smooth vector fields on  $X_1,\dots, X_m$, we define the Carnot-Carath\'eodory distance as the minimal length distance (or geodesic distance) among all horizontal curves joining $x$ to $y$, that is
$$
d_{CC}(x,y)=\inf\left\{
\int_0^T|\alpha^{\xi}(t)|\,dt\,\bigg|\, \xi(0)=x,\,\xi(T)=y\, \textrm{and $\xi$ is horizontal}
\right\},
$$
where $|\alpha^{\xi}(t)|$ is the Euclidean norm of the $m$-valued horizontal velocity.
\end{defi}
Whenever $X_1,\dots, X_m$ satisfy the H\"ormander condition (as in our  case of Carnot groups), then
$d_{CC}(x,y)<+\infty$ for all $x,y\in\re^N$ and it is continuous w.r.t. the Euclidean topology on $\R^N$.\\
 We denote by $\|x\|_{CC}:=d_{CC}(x,0)$ the Carnot-Carath\'eodory norm.\\

\begin{rem}
The Carnot-Carath\'eodory distance  is globally equivalent to the so-called minimal-time (or control) distance that is defined as  
$$\hat{d}(x,y):=\inf\{T\geq 0 | \exists\ \xi\,\textrm{subunit horizontal  in}\, [0,T]\ \textrm{with}\; \xi(0)=x, \xi(T)=y\},$$
where an absolutely continuous  curve $\xi:[0,T]\to \R$ is called {\em subunit horizontal} if satisfies \eqref{EQ_Horizontal} and $|\alpha^{\xi}(t)|\leq 1$ for a.e. $t\in [0,T]$.\\
\end{rem}

Note that, even if it is possible to give an explicit formulation for the Carnot-Carath\'eodory distance in $\dH$ (by computing the geodesics), this is extremely complicated so we omit that.\\
Thus we will need to use both the Carnot-Carath\'eodory distance and the homogeneous distance, so it is important to recall the relation between these distances and between them and the standard Euclidean distance in $\R^N$.
\begin{lemma}
\label{ReelationDistances}
Let $d_h$ and $d_{CC}$ be  the homogeneous distance  and the Carnot-Carath\'eodory distance defined respectively in Definitions \ref{homogenousNorm} and \ref{CC-distance_definiton}.
Then for any compact $K\subset\re^N$ there exists a positive constant $C_K$ such that
$$ C^{-1}_K|x-y|\leq d_{CC}(x,y)\leq C_K |x-y|^{1/r},$$
where $r$ is the step of the Carnot group and $|x-y|$ denotes here the standard Euclidean distance in $\R^N$.\\
The same statement holds also replacing $d_h$ and $d_{CC}$.\\
Moreover $d_h$ and $d_{CC}$ are equivalent distance on compact sets, i.e. for any compact $K\subset\re^N$ there exists a positive constant $c_K$ such that
$$ c^{-1}_Kd_{h}(x,y)\leq d_{CC}(x,y)\leq c_K d_{h}(x,y).$$
\end{lemma}
For the proof we refer to the monograph \cite{BLU}.\\

In the following Lemma we collect several properties of horizontal curves that will be very useful later.
\begin{lemma}\label{leftinv}
Let 
$\xi$ be a horizontal curve with velocity 
$\alpha^{\xi}(s)$
such that 
$\xi(0)=x$ and $\xi(t)=y$.
Then the following properties hold:
\begin{enumerate}
\item[(i)]  For any $z\in\re^N$, $\widetilde\xi(s):=z\circ\xi(s)$ is still horizontal 
with
$\alpha^{\widetilde\xi}(s)=\alpha^{\xi}(s)$,
$\widetilde\xi(0)=z\circ x$ and $\widetilde\xi(t)=z\circ y$.
\item[(ii)]   For any $C>0$, $\eta(s):=\xi(Cs)$ is still horizontal 
with
$\alpha^{\eta}(s)=C\alpha^{\xi}(Cs)$, 
$\eta(0)=x$ and $\eta(t/C)=y$.
\item[(iii)]  For any $\lambda>0$, $\hat\xi(s):=\delta_\lambda(\xi(s))$
 is still horizontal 
with
$\alpha^{\hat\xi}(s)=\lambda\alpha^{\xi}(s)$.
\end{enumerate}
\end{lemma}
\begin{proof}
\begin{enumerate}
\item[(i)] Denote by $L_z$ the left translation w.r.t. $z$ (i.e. $L_z(x)=z\circ x$)  and by $DL_z$ the differential of the left translation $L_z$. We have 
\begin{eqnarray*}
\dot{\widetilde\xi}(s)&=&DL_z(\xi(s))\;\dot{\xi}(s)=
DL_z(\xi(s))\;\bigg(
\sum_{i=1}^{m}\alpha_i^{\xi}(s)X_i(\xi(s))\bigg)=\\
&=&\sum_{i=1}^{m}\alpha_i^{\xi}(s)DL_z(\xi(s))\;X_i(\xi(s))=
\sum_{i=1}^{m}\alpha_i^{\xi}(s)X_i(z\circ\xi(s))=\\
&=& \sum_{i=1}^{m}\alpha_i^{\xi}(s)X_i(\widetilde\xi(s)),
\end{eqnarray*}
where we have used the fact that the vector fields $X_i$ are left-invariant by definition, i.e.
 $DL_z(\xi(s))\; X_i(\xi(s))=X_i(z\circ\xi(s))$, for all $z$.\\
\item[(ii)]  
For any $C\in\re$,  given $\eta(s)=\xi(Cs)$, then
\begin{equation}
\label{cambiooriz}
\dot{\eta}(s)=C\dot{\xi}(C s)= C \bigg(\sum_{i=1}^{m}\alpha_i^{\xi}(C s)
X_i(\xi(C s))\bigg)=\sum_{i=1}^{m}C\alpha^{\xi}_i(C s)
X_i(\eta(s)),
\end{equation}
so $\eta$ is horizontal with $\alpha^{\eta}(s)=C\alpha^{\xi}(Cs)$.\\
\item[(iii)]  
Using the fact that we are in exponential coordinates and the definition of dilations as automorphisms of the group by the exponential map, that is:
$$
\delta_{\lambda}\left(
\textrm{exp}\left(
\sum_{i=1}^r\sum_{j=1}^{m_i}g_{j,i} X_{j,i}
\right)
\right)
=
\textrm{exp}
\left(
\sum_{i=1}^r\sum_{j=1}^{m_i}\lambda^ig_{j,i} X_{j,i}
\right),
$$
where  $X_{j,i}$ for $j=1,\dots, m_i$ are a basis for the layer $V_i$,
and 
$g_{j,i}$ are the associated exponential coordinates for the point $g\in \G=\R^N$. \
From the previous formula written for horizontal curves, that means  $i=1$ and $j=1,\dots, m_1=m$, it follows immediately that 
$\hat\xi(s):=\delta_\lambda(\xi(s))$
 is horizontal 
and
$\alpha^{\hat\xi}(s)=\lambda\alpha^{\xi}(s)$.
\end{enumerate}
\end{proof}
The following lemma proves that we can control the supremum norm of two curves by the $L^1$-norm of the associated horizontal velocity.

\begin{lemma} 
 \label{aprroxCurveLemma}
 Consider two measurable functions $\alpha,\beta:[0,T]\to \R^m$ and the associated horizontal curves $\xi^{\alpha},\xi^{\beta}$ starting from the same initial point, i.e. 
$$
\dot{\xi}^{\alpha}(s)=\sum_{i=1}^m \alpha_i (s)X_i\big(\xi^{\alpha}(s)\big),\quad
\dot{\xi}^{\beta}(s)=\sum_{i=1}^m \beta_i (s)X_i\big(\xi^{\beta}(s)\big),
\quad \xi^{\alpha}(0)=\xi^{\beta}(0).
$$
If $\alpha,\beta$ are equi-bounded  in $L^1(0,T)$, then there exists a positive constant $C>0$  such that
$$
\norma{\xi^\alpha-\xi^\beta}_{\infty}\leq C  \norma{\alpha-\beta}_{L^1(0,T)}.
$$
 \end{lemma}
 \begin{proof}
The proof is trivial. In fact:
$$
\xi^\alpha(t)-\xi^\beta(t)=\int_0^t \big[\dot{\xi}^\alpha(s)-\dot{\xi}^\beta(s)\big]d s=
\sum_{i=1}^m\int_0^t \left[ \alpha_i (s)X_i\big(\xi^{\alpha}(s)\big)- \beta_i (s)X_i\big(\xi^{\beta}(s)\big) \right] d s
$$
At this point we add $\pm  \alpha_i (s)X_i\big(\xi^{\beta}(s)\big)$, we use that the vector fields are smooth 
(so in particular locally Lipschitz continuous) and that $\alpha,\beta$ are equibounded. 
$\xi^{\alpha},\xi^{\beta}$ are equi-bounded, 
Hence by Gronwall's inequality one can easily conclude.

 \end{proof}
 The manifold structure is crucial when one works with PDEs. In fact vector fields allow us to define naturally derivatives of any order, just considering how a vector field acts on smooth functions.  Since we are interested in first-order Hamilton-Jacobi equations we introduce only the first derivatives.
 \begin{defi}[Horizontal gradient] For a Carnot group defined on $\R^N$, consider the family of vector fields $\cX=\{X_1,\dots,X_m\}$ associated to the Carnot group.
 The horizontal gradient is defined as 
 $$
 \cX u:=\big(X_1u\big)X_1+\dots +\big(X_mu\big)X_m.
 $$
\end{defi}
\begin{rem}
 Note that $\cX u$ is always an element of the distribution $\cH$ since it is defined as a linear combination of the vector fields that span the distribution.\\
 For sake of simplicity we will often identify the horizontal gradient (which is a $N$-dimensional object in $\cH$) with its  coordinate vector $\nabla_{\cX} u=(X_1 u,\dots ,X_mu)^T$ which is instead an element in $\R^m$.
Trivially $\nabla_{\cX}u=\sigma(x) \nabla u$ where $\nabla u$ denotes the standard (Euclidean) gradient of $u$ and 
$\sigma$ is defined in \eqref{matrixC}.
\end{rem}
\begin{ex}
In the case of $\dH$ the horizontal gradient can be explicitly  written as
$$
\nabla_{\cX} u=\begin{pmatrix}
u_{x_1}-\frac{x_2}{2}u_{x_3}\\
u_{x_2}+\frac{x_1}{2} u_{x_3}
\end{pmatrix}
=
\begin{pmatrix}
1 &0& -\frac{x_2}{2}\\
0&1&\frac{x_1}{2}
\end{pmatrix}\nabla u
\in \R^2,
$$
while $\cX u=\left(u_{x_1}-\frac{x_2}{2}u_{x_3}
\right) X_1(x_1,x_2,x_3)+\left(u_{x_2}+\frac{x_1}{2}u_{x_3}
\right) X_2(x_1,x_2,x_3)
\in \R^3
$.
\end{ex}



\section{Statement of the problem.}
For any $\e>0$, we look at the following family of randomly perturbed  problems:
\begin{equation}
\label{ApproxPr}
\left\{
\begin{aligned}
&
u^{\varepsilon}_t+H\left(\delta_{1/\varepsilon}(x), \sigma(x)\nabla u^{\varepsilon},\omega\right)=0,\ x\in\re^N, \;\omega \in \Omega, \ t>0,\\
& u^{\varepsilon}(0,x,\omega)=g(x),\ x\in \re^N, \;\omega \in \Omega,
\end{aligned}\right.
\end{equation}
where $\big(\Omega,\cF,\dP\big)$ is a probability space, $\sigma(x)$ is a smooth $m\times N$ matrix (with $m\leq N$), whose columns are vector fields $X_1, X_2,\dots,X_m$ associated to a Carnot group, $\delta_{\lambda}(\cdot)$ are the anisotropic dilations defined by  the Carnot group structure.\\
We assume that  
the Hamiltonian $H:\R^N\times \R^m\times \Omega\to \R$ satisfies the following assumptions w.r.t. $x\in \R^N$, $q=\sigma(x)p\in \R^m$ and $\omega\in \Omega$:
\begin{description}
\item[{\bf (H1)}]  $ q\mapsto H(x,q,\omega)$ is convex in $q$;
\item[{\bf (H2)}]  $\exists \; \overline{C}_1>0, \; \overline{\lambda}>1$ such that
$$ \overline{C}^{-1}_1(|q|^{\overline{\lambda}}-1)
 \leq H(x,q,\omega)
 \leq \overline{C}_1(|q|^{\overline{\lambda}}+1), \; \forall\; (x,q,\omega)\in \R^N\times\R^m\times \Omega;
 $$
\item[{\bf(H3)}] there exists $\overline{m}:[0,+\infty)\to [0,+\infty)$ concave, monotone increasing
 with $\overline{m}(0^+)=0$ and such that for all $x,y\in \R^N,q\in \R^m,\omega\in \Omega$
$$
 |H(x,q,\omega)-H(y,q,\omega)|\leq \overline{m}\big(\|-x\circ y\|_h(1+|q|^{\overline{\lambda}})\big);$$
\item[{\bf(H4)}] $ \forall\, q\in \R^m$ the function $(x,\omega)\mapsto H(x,q,\omega)$ 
is a stationary, ergodic random field on $\R^N\times \Omega$ w.r.t. the unitary translation operator, associated to 
the Carnot group structure.
\end{description}
Recall that $|q|$ is the usual Euclidean norm in $\R^m$ while $\norma{x}_h$ is the homogeneous distance  in Definition \ref{homogenousNorm}, and $x^{-1}=-x$  in exponential coordinates.
\begin{ex}[Main model]
\label{exampleMainModels} {\rm
The main model that we have in mind is:
\begin{equation}
\label{model1}
 H(x,q,\omega)= a(x,\omega)\frac{|q|^{\beta}}{\beta} +V(x,\omega),
\end{equation}
with  $\beta>1$ and $V$ and $a(x,\omega)$  are bounded, 
 and satisfying {\bf(H3)} and  {\bf(H4)} and  $a$ is also uniformly strictly positive.
}\end{ex}

\begin{rem} [Non-coercivity of the Hamiltonian] Note that the main difference between these assumptions and the assumptions in \cite{S1} 
is that {\em the Hamiltonian is not anymore coercive  w.r.t. the gradient but only w.r.t. the lower dimensional horizontal gradient} (assumption {\bf(H2)}).
This lack of coercivity w.r.t.  the total gradient variable $p$ is what will make the approach in \cite{S1} and \cite{RT} failing and will lead to the main technical difficulties.
\end{rem}
\begin{rem}
Assumption {\bf(H3)} is adapted to the anisotropic dilations in the group. Nevertheless using Lemma \ref{ReelationDistances} this assumption can be rewritten in terms of the standard Euclidean distance (with a power depending on the step of the group). In particular if $H(x,p,\omega)$ is Lipschitz continuous in $x$ w.r.t. the homogeneous distance then it is only H\"older continuous  in $x$ w.r.t. the standard Euclidean distance with power $1/r$ and $r=$step of the Carnot group. E.g. in the Heisenberg group it would be $1/2$-H\"older continuous.
\end{rem}

We next write more explicitly assumption {\bf (H4)} to show how this adapts to the algebraic structure of the Carnot group.\\
Assumption {\bf (H4)}  means that there exists a family of measure-preserving maps $\tau_x:\Omega\to \Omega,$ indexed by either $x$ in the Carnot group or $x$ in a discrete version of the Carnot group ($\mathbb Z^N$ as subset of the Carnot group equipped with the group operation of the Carnot group) with the following properties:
\begin{itemize}
\item $\tau_0=id$
\item $\tau_x(\tau_y(\omega))=\tau_{x\circ y}(\omega)$  for all $x,y\in \mathbb R^N$ (or $\mathbb Z^N$, respectively) and almost all $\omega\in \Omega.$
\item (Stationarity) $H(x,q,\omega)=H(0,q,\tau_x(\omega))$   for all $x\in \mathbb R^N$ (or $\mathbb Z^N$, respectively) and almost all $\omega\in \Omega.$
\item (Ergodicity) If $A\subseteq\Omega$ is such that $\tau_x(A)=A$ for all $x\in \mathbb R^N$ (or $\mathbb Z^N$, respectively), then ${\mathbb P}(A)\in \{0,1\}.$
\end{itemize} 

Examples are short-correlated Euclidean-stationary random fields (by Borel-Cantelli) or (for the discrete Heisenberg group) Heisenberg-periodic sets where an independent identically distributed random variable is chosen for each cell.

\begin{rem}We would like to add some remarks concerning the assumptions on ergodicity and stationarity.
\begin{enumerate}
\item Note that we apply the ergodic theorem only to a one-dimensional {\em Abelian} subgroup of the Carnot group ($\cX$-lines), so for convergence we are completely in  the framework of classical ergodic theory, see proof of Theorem \ref{convzero}. The ergodicity with respect to actions of the full group is only used to establish that the limit is deterministic.
\item In case of the Example \ref{model1},  and short-correlated random coefficients, the convergence to a deterministic limit in Theorem \ref{convzero} follows already from the law of large numbers.
\item For examples of ergodic actions of the Heisenberg group on a probability space $(\Omega, {\mathcal F}, \dP)$ see e.g. $\cite{Da}.$ In order to obtain from this a model like Example \ref{model1} with $a(x,\omega)= 1$,
take a bounded random variable $V:\Omega\to \R$ and set $V(x,\omega)= V(\tau_x(\omega))$.
\end{enumerate}
\end{rem}
\begin{ex} 
An explicit example for a model satisfying {\bf (H4)} in the special case of Example \ref{model1} and in Heisenberg group $\dH$ can be constructed in the following way:

\noindent
Take three independent random fields on $\R,$ $f_i(x,\omega):\R\times \Omega\to \R,$ for $i=1,2,3,$ such that they are stationary ergodic w.r.t.
the action of $\R.$  Then for a Borel-measurable bounded function $G:\R^3\to \R$ the random potential 
$$ V(x_1,x_2,x_3,\omega):=G(f_1(x_1,\omega),f_2(x_2,\omega), f_3(x_3,\omega))$$ 
is Heisenberg-stationary. Indeed, by independence and one-dimensional stationarity we have for any open intervals $(a_1,a_2), (b_1,b_2), (c_1,c_2)$ that
\begin{eqnarray*}
&&{\mathbb P}\{(x_1+r,x_2+s,x_3+t+1/2(x_1s-x_2r))\in (a_1,a_2)\times(b_1,b_2)\times(c_1,c_2)\}\\&& =
{\mathbb P}\{f_1(x_1+r)\in (a_1,a_2)\}{\mathbb P}\{f_2(x_2+s)\in (b_1,b_2)\}\\&& \times{\mathbb P}\{f_3(x_3+t+1/2(x_1s-x_2r))\in(c_1,c_2)\}\\ &&=
{\mathbb P}\{f_1(x_1)\in (a_1,a_2)\}{\mathbb P}\{f_2(x_2)\in (b_1,b_2)\}{\mathbb P}\{f_3(x_3)\in(c_1,c_2)\}\\ &&={\mathbb P}\{(x_1,x_2,x_3)\in (a_1,a_2)\times(b_1,b_2)\times(c_1,c_2)\}.
\end{eqnarray*}
Since open rectangles generate the Borel-$\sigma$-algebra, the result follows. 
\end{ex}

We introduce 
\begin{equation}\label{rapprepsilon}
u^{\varepsilon}(t,x,\omega):=\inf\limits_{y\in \re^N}\big\{g(y)+L^{\varepsilon}(x,y,t,\omega)\big\}
\end{equation}
with
\begin{equation}\label{Lepsilon}
L^{\varepsilon}(x,y,t,\omega):=\inf\limits_{\xi\in \cA^t_{y,x}} \int_0^t H^*\left(\Del(\xi(s)), \alpha^{\xi}(s),\omega\right)ds
\end{equation}
where 
$$
\cA^t_{y,x}:=\big\{\xi\in W^{1,\infty}\big((0,t)\big)\,\big|\, \xi \ \;\textrm{horizontal curve such that}\; \xi(0)=y\;\textrm{and}\;\xi(t)=x\big\},
$$
while $H^*$ denotes the Legendre-Fenchel transform of $H$ w.r.t. $q\in \R^m$, that is 
$$H^*(x,q,\omega):=\sup_{p\in \R^m}\{p\cdot q-H(x,p,\omega)\}$$ and 
$\alpha^{\xi}(s)$ is the $m$-valued measurable function corresponding to the horizontal  velocity of $\xi(s)$ defined in \eqref{EQ_Horizontal}.

\begin{theorem}\label{th1}
We assume $g$  bounded and Euclidean Lipschitz continuous,
 and that the Hamiltonian $H$ satisfies {\bf(H1)-(H3)}. Then, for all fixed $\omega\in \Omega$, $u^{\varepsilon}(t,x,\omega)$ given by formula \eqref{rapprepsilon} is the  unique $BUC$  viscosity solution of problem \eqref{ApproxPr}.
 \end{theorem}	
The proof of this theorem will be given in the  Appendix.\\
From now on we use the notation:
\begin{equation}
\label{Lagrangian}
L(x,q,\omega):=H^*(x,q,\omega)=\sup_{p\in \R^m} \{ p\cdot q-H(x,p,\omega)\};
\end{equation}
$
L(x,\cdot,\omega)
$ is the Legendre-Fenchel transform of the Hamiltonian $H(x,\cdot ,\omega)$ taken w.r.t. $q\in\R^m$, for each $x\in \R^N$ and $\omega\in \Omega$ fixed, and it is called {\em Lagrangian} associated to the given Hamiltonian.
In the following lemma we show how the properties of the Hamiltonian pass to the associated Lagrangian.

\begin{lemma}\label{PropertiesLagrangian} 
If $H(x,q,\omega)$ satisfies {\bf (H1)-(H4)} and $
L(x,q,\omega)
$ is the associated Lagrangian defined by \eqref{Lagrangian}, then $L$ satisfies
\begin{description}
\item[{\bf (L1)}]
$ q\mapsto L(x,q,\omega)$ is convex, for all $(x,\omega)\in \R^N\times \Omega$;
\item[{\bf  (L2)}]
there exists $C_1>0$ such that
$$
C^{-1}_1(|q|^{\lambda}-1)
 \leq L(x,q,\omega)
 \leq C_1(|q|^{\lambda}+1),
 \quad\forall\; (x,q,\omega)\in \R^N\times\R^m\times \Omega,$$
 where $\lambda=\overline{\lambda}^*:=\frac{\overline{\lambda}}{\overline{\lambda}-1}$ (i.e., the conjugate of the constant $\overline{\lambda}$ defined in assumption {\bf (H2)});
\item[{\bf (L3)}] for all $R>0$, there exists $ m_R:[0,+\infty)\to [0,+\infty)$ 
concave, monotone increasing, with $m_R(0^+)=0$ and such that for all
$x,y\in \R^N,\omega\in \Omega$
 $$   |L(x,q,\omega)-L(y,q,\omega)|\leq 
 m_R\big(\|-x\circ y\|_h\big),\quad \forall\;q\in B_R(0),
 $$
 where $B_R(0)$ is a (Euclidean) ball in $\R^m$ of radius $R$;
 \item[{\bf (L4)}]
$\forall\, q\in \R^m$ the function $(x,\omega)\mapsto L(x,q,\omega)$
is stationary, ergodic random 
field on, $\R^N\times \Omega$ w.r.t. the unitary translation operator associated to 
the Carnot group structure;
\item[{\bf (L5)}]
$ L^*=H$.
\end{description}
\end{lemma}
\begin{proof}  Properties {\bf (L1)} and {\bf (L5)} follow immediately by  definition of $L=H^*$.
Property {\bf (L2)} comes trivially from {\bf (H2)} and \eqref{Lagrangian} taking $\lambda=\overline{\lambda}^{\;*}$, the conjugate exponent of $\overline \lambda$.
The proofs of  {\bf (L3)}  and {\bf (L4)} are also immediate: one can find a detailed proof of  {\bf (L3)}  in \cite[Theorem A.2.6]{CS} while {\bf (L4)} comes directly by the definition of $L=H^*$.
\end{proof}
\begin{ex}
\label{ExampleModelsLagrangian}
In the case of the model \eqref{model1}  the associated Lagrangian is 
\begin{equation}
\label{model1Lagrangian}
L(x,q,\omega)=
b(x,\omega)\frac{|q|^{\beta^*}}{\beta^*}+V(x,\omega),
\end{equation}
with $\beta^*=\frac{\beta}{\beta-1}$ and $b(x,\omega)=\frac{1}{a(x,\omega)^{\frac{1}{\beta -1}}}.$
\end{ex}

Here we state the main results of this paper. The proof will be given in Section \ref{HomogSection}.
\begin{theorem}\label{mainTH}
Consider the problem \eqref{ApproxPr} with  $g:\R^N\to \R$ bounded and Lipschitz continuous. Assume that the Hamiltonian 
$H(x,p,\omega)$ satisfies assumptions {\bf(H1)-(H4)} and for $L=H^*$
the following additional assumption holds:
\begin{description}
\item[{(L6)}] there exist  constants $ C>0$ and $ \lambda>1$ such that
$$
|L(x,p,\omega)-L(x,s\,p,\omega)|
\le C\big|1-|s|^\lambda\big| \,|L(x,p,\omega)|,
$$
for all $ s\in \R,\;x\in \R^N, p\in\R^m, \; \omega\in\Omega$.
\end{description}
Then the viscosity solutions $u^{\varepsilon}(t,x,\omega)$ of problem~\eqref{ApproxPr} converge locally uniformly in $x$ and $t>0$ and a.s. in $\omega\in \Omega$ to the unique solution $u(t,x)$ of the deterministic problem \eqref{LimitProblem}, where  the effective Hamiltonian $\overline H(q)$ is defined as $\overline H(q)= \overline L ^*(q)$ and $\overline L(q)$ is the effective Lagrangian defined by limit \eqref{EffectiveLagrangian}.
\end{theorem}
Now we give a class of operators where we can apply the previous result.
\begin{corollary}\label{cor}
Consider the problem \eqref{ApproxPr} with  $g:\R^N\to \R$ bounded and Lipschitz continuous. Assume that the Hamiltonian 
$H(x,p,\omega)$ satisfies assumptions {\bf(H1)-(H4)} and moreover 
\begin{equation}
\label{LambdaHomogeneity}
H(x,p,\omega)=H_1(x,p,\omega)+H_2(x,\omega),
\end{equation}
with $H_1(x,p,\omega)$ $\lambda$-homogeneous in $p$ (in the sense that $H_1(x,s\,p,\omega)=|s|^{\lambda}H^1(x,p,\omega)$). Then
the viscosity solutions $u^{\varepsilon}(t,x,\omega)$ of \eqref{ApproxPr} converge locally uniformly in $x$ and $t>0$ and a.s. in $\omega\in \Omega$ to the unique solution $\bar u(t,x)$ of the deterministic problem \eqref{LimitProblem}, where  the effective Hamiltonian $\overline H(q)$ is defined as $\overline H(q)= \overline L ^*(q)$ and $\overline L(q)$ is the effective Lagrangian defined by limit \eqref{EffectiveLagrangian}.
   \end{corollary}
   \begin{proof}[Corollary \ref{cor}]
    We need only to remark that, whenever $H$ satisfies \eqref{LambdaHomogeneity}, then the associated Lagrangian has the same structure (by taking $\lambda=\overline{\lambda}^*$), and such a structure for $L$ implies  assumption {\bf (L6)}. 
   Hence Theorem \ref{mainTH} immediately implies the result.
   \end{proof}
\begin{example}
In particular our main model of Hamiltonian~\eqref{model1}  satisfies the assumptions of  Corollary \ref{cor}.
   \end{example}

\section{Properties of $L^{\varepsilon}(x,y,t,\omega)$.}
\label{SectionEstimates}
In this section we will investigate several properties for the variational problem $L^{\varepsilon}(x,y,t,\omega)$ defined by \eqref{Lepsilon}.
\begin{lemma}\label{lemma71}
Under assumption {\bf(L2)}
 then
\begin{equation*}
L^\varepsilon(x,y,t,\omega)\leq C_1 t +C_1\frac{d_{CC}(x,y)^\lambda}{t^{\lambda-1}}, \qquad \forall x,y\in \R^N,t>0,\omega\in \Omega,
\end{equation*}
where $C_1$ and $\lambda$ are the constants introduced in assumption {\bf(L2)}.
\end{lemma}
\begin{proof}
We consider a geodesic $\eta$ from $y$ to $x$ in time $t$ (parametrized by arc-length), then $\eta\in {\cal A}^t_{y,x}$ with $|\alpha^\eta(s)|=d_{CC}(x,y)/t$. By assumption {\bf(L2)} we have
\[
L^\varepsilon(x,y,t,\omega)\leq\int_0^t L(\delta_{1/\varepsilon}(\eta(s),\alpha^\eta(s), \omega)\, ds\leq \int_0^t C_1(1+d_{CC}(x,y)^\lambda t^{-\lambda})\, ds,
\]
which easily entails the statement.
\end{proof}
\begin{proposition}\label{prp71}
Under assumption {\bf(L2)} then
 \begin{equation}
L^\varepsilon(x,y,t,\omega)=\inf\limits_{\eta}\left\{\int_0^t L(\delta_{1/\varepsilon}(\eta(s)),\alpha^\eta(s),\omega)\, ds \right\},
 \end{equation}
 where the infimum is taken over the curves
 \begin{equation}\label{7.1}
   \eta\in{\cal A}^t_{y,x} \qquad \textrm{such that }\|\alpha^\eta\|_{L^\lambda(0,t)}\leq \widetilde{C},
\end{equation}   
 where $\widetilde{C}= [(C_1^2+C_1+1)t^\lambda +C_1^2 d_{CC}(x,y)^\lambda]^{1/\lambda}t^{\frac{1}{\lambda}-1}$.\end{proposition}
\begin{proof}
  Using the definition of $L^\varepsilon(x,y,t,\omega)$ as greatest lower bound, then there exists  a curve $\eta\in{\cal A}^t_{y,x}$ 
  such that 
\[
\int_0^t L(\delta_{1/\varepsilon}(\eta(s)), \alpha^\eta(s),\omega)\, ds\leq L^\varepsilon(x,y,t,\omega) +t.
\]  
By the bound from below in {\bf(L2)} and using Lemma~\ref{lemma71}, we have
\begin{equation*}
C_1^{-1}\int_0^t\left(|\alpha^\eta (s)|^\lambda-1 \right)ds\leq L^\varepsilon(x,y,t,\omega) +t
\leq (C_1+1)t +C_1 \frac{d_{CC}(x,y)^\lambda}{t^{\lambda-1}},
\end{equation*}  
and consequently
\begin{equation*}
\|\alpha^\eta\|_{L^\lambda(0,t)}^\lambda\leq (C_1^2+C_1+1)t +C_1^2 \frac{d_{CC}(x,y)^\lambda}{t^{\lambda-1}},
\end{equation*}
which is equivalent to relation~\eqref{7.1}.
 \end{proof}

\begin{corollary} 
\label{corollary71}
Under assumption {\bf (L2)} then the infimum in $L^{\varepsilon}(x,y,t,\omega)$ is attained over  admissible curves $\eta$ such that
 \begin{description}
\item[(i)]  $ \|\alpha^\eta\|_{L^{\gamma}(0,t)}\leq C_2$, for all $ 1\leq \gamma\leq \lambda$ and  for all $t<T$, where $\lambda$ is the constant given in assumption {\bf (L2)} and $C_2$  depends only on  the constants in assumption {\bf(L2)} and the bound $T$,
\item[(ii)]$ \|\eta\|_{\infty}\leq  C_3$, where  $C_3$  depends only on $\sigma(x)$, the constants in assumption {\bf(L2)} and $T$ for all $t<T$.
 \end{description}

 \end{corollary}
\begin{proof} This follows easily by Proposition \ref{prp71}, the standard embedding properties of $L^\lambda (0,t)$ and 
$\xi(t)=\xi(0)+\int_0^t\sigma(\xi(s))\alpha^\xi(s) ds.$
\end{proof}

We now prove the uniform continuity of $L^\varepsilon$, uniformly w.r.t. $\varepsilon>0$. To this purpose, we adapt some arguments from  \cite{RT}. 

\begin{lemma}\label{lemma41}
Under assumption {\bf(L2)}, then
\begin{equation}\label{43}
L^{\varepsilon}(x,y,t+h,\omega)\leq L^{\varepsilon}(x,y,t,\omega)+C_1h,  
 \end{equation}
for all 
$ \varepsilon>0, x,y\in \R^N,t>0,\,\omega\in \Omega$ and where $C_1$ is the constant introduced in {\bf(L2)}.
\end{lemma}
\begin{proof}
We  proceed as in \cite{RT}.
Let $\xi$ be an admissible path for $L^{\varepsilon}(x,y,t,\omega)$; we introduce 
\[
\xi_1(s):=\left\{\begin{array}{ll}
\xi(s) &0\leq s\leq t,\\
x&t\leq s\leq t+h.
\end{array}\right.
\]
Note that $\xi_1$ is an admissible path for $L^{\varepsilon}(x,y,t+h,\omega)$. Hence, we have
\begin{eqnarray*}
L^{\varepsilon}(x,y,t+h,\omega)&\leq &\int_0^{t+h}L(\delta_{1/\varepsilon}(\xi_1(s)), \alpha^{\xi_1}(s),\omega)ds\\
&\leq & \int_0^tL(\delta_{1/\varepsilon}(\xi(s)), \alpha^\xi(s),\omega)ds +\int_t^{t+h}L(\delta_{1/\varepsilon}(x), 0,\omega)ds.
\end{eqnarray*}
Taking the infimum over $\xi$ and by assumption {\bf(L2)}, we obtain the statement. 
\end{proof}
\begin{lemma}\label{lemma2}
Under assumption {\bf(L2)}, then
\begin{equation}\label{4.2}
  L^{\varepsilon}(x\circ v,y,t+\|v\|_{CC},\omega) \leq L^{\varepsilon}(x,y,t,\omega)+ 2C_1 \|v\|_{CC},
\end{equation}
for all 
$ \varepsilon>0, x,y,v\in \R^N,t>0,\,\omega\in \Omega$ and
where $C_1$ is the constant introduced in assumption {\bf(L2)}.
Moreover, for each compact $K\subset \subset \R^N$, there exists a constant $C$ (depending only on $K$ and on the assumptions of the problem, so in particular independent of $\varepsilon$) such that 
\begin{equation}\label{4.3}
  L^{\varepsilon}(x\circ v,y,t+\|v\|_{h},\omega) \leq L^{\varepsilon}(x,y,t,\omega)+ C \|v\|_{h}\qquad \forall v\in K,
\end{equation}
and for all 
$ \varepsilon>0, x,y\in \R^N,t>0,\,\omega\in \Omega$.

\end{lemma}
\begin{proof}
For any curve $\xi \in {\cal A}^t_{y,x}$, we define
\[
\eta_\xi(s):=\left\{\begin{array}{ll}
\xi(s)&\quad 0\leq s\leq t\\
\tilde\gamma (s)&\quad t\leq s\leq t+\|v\|_{CC}
\end{array}\right.
\]
where $\tilde\gamma (s)=\gamma(s-t)$ and $\gamma$ is a geodesic from $x$ to $x\circ v$ in time $\|v\|_{CC}$.
Since $\eta_\xi\in {\cal A}^{t+\norma{v}_{CC}}_{y,x\circ v}$,  then we have
\begin{eqnarray*}
&& L^{\varepsilon}(x\circ v,y,t+\|v\|_{CC},\omega) \leq
 \int_0^{t+\|v\|_{CC}}L(\delta_{1/\varepsilon}(\eta_\xi(s)),\alpha^{\eta_\xi}(s),\omega)ds\\
 &&\qquad = \int_0^{t}L(\delta_{1/\varepsilon}(\xi(s)),\alpha^\xi(s),\omega)ds+
 \int_t^{t+\|v\|_{CC}}L(\delta_{1/\varepsilon}(\tilde \gamma(s)),\alpha^{\tilde \gamma}(s),\omega)ds\\
&&\qquad \leq \int_0^{t}L(\delta_{1/\varepsilon}(\xi(s)),\alpha^\xi(s),\omega)ds+ 2C_1 \|v\|_{CC},
\end{eqnarray*}
where the last inequality is due to assumption {\bf(L2)} and  the fact that by arc-length parametrisation we can assume $|\alpha^{\tilde\gamma}(s)|=1$. Taking the infimum over $\xi$, we get the  bound \eqref{4.2}.

It remains to prove \eqref{4.3}.
We argue as before defining
\[
\eta_\xi(s):=\left\{\begin{array}{ll}
\xi(s)&\quad 0\leq s\leq t\\
\tilde\gamma_1 (s)&\quad t\leq s\leq t+\|v\|_{h}
\end{array}\right.
\]
where $\tilde\gamma_1 (s):=\gamma_1 (s-t)$ and $\gamma_1$ is a geodesic from $x$ to $x\circ v$ in time $\|v\|_{h}$. We have $|\alpha^{\gamma_1}(s)|=\|v\|_{CC}/\|v\|_{h}$. Hence, as before, we get
\begin{equation*}
L^{\varepsilon}(x\circ v,y,t+\|v\|_{h},\omega) \leq
\int_0^{t}L(\delta_{1/\varepsilon}(\xi(s)),\alpha^\xi(s),\omega)ds+
C_1\|v\|_{h}\left(1+\frac{\|v\|_{CC}^\lambda}{\|v\|_{h}^\lambda}\right).
\end{equation*}
Recall that for each compact $K\subset\subset\R^N$, there exists a constant $c$ such that $c^{-1}\|v\|_{h}\leq \|v\|_{CC}\leq c\|v\|_{h}$ (see Lemma \ref{ReelationDistances}); hence
\[
L^{\varepsilon}(x\circ v,y,t+\|v\|_{h},\omega) \leq
\int_0^{t}L(\delta_{1/\varepsilon}(\xi(s)),\alpha^\xi(s),\omega)ds+ C_1(1+c^\lambda)\|v\|_{h}
\]
and, taking the infimum over $\xi$, we get \eqref{4.3}.
\end{proof}

\begin{lemma}\label{prop1} Under assumptions {\bf(L2)} and {\bf(L6)}, then
$L^{\varepsilon}(x,y,t,\omega)$ is locally uniformly continuous in $t$ away from $0$, locally uniformly w.r.t. $x$ and $y$ and uniformly w.r.t. $\varepsilon$. More precisely
for any $\delta\in(0,1)$ there exists  $C_{\delta}>0$
  (depending only on the constants in assumptions {\bf(L2)} and {\bf(L6)} and going to $+\infty$ as  $\delta\to 0^+$)  such that
\begin{equation}\label{contunift}
|L^{\varepsilon}(x,y,t+h,\omega)-L^{\varepsilon}(x,y,t,\omega)|\leq C_{\delta}h\quad \forall \varepsilon>0,
\end{equation}
and for any $t, t+h \in [\delta, 1/\delta]=:I_{\delta}$, $h>0$ and for any 
$(x,y) \in A_{\delta}$ where  $$A_{\delta}:=\{ (x,y)\in\re^N\times \R^N\,|\, d_{CC}(x,y)<1/\delta\}.$$
\end{lemma}
\begin{proof}
By Lemma \ref{lemma41} we know that for any $h>0$
$$L^{\varepsilon}(x,y,t+h,\omega)-L^{\varepsilon}(x,y,t,\omega)\leq C_1h.$$
It remains to show the opposite inequality, i.e. 
\begin{equation}\label{9settembre}
L^{\varepsilon}(x,y,t+h,\omega)-L^{\varepsilon}(x,y,t,\omega)\geq C_{\delta}h.
\end{equation}  
Take $(x,y)\in A_{\delta}$ and a curve $\eta$ admissible for $L^{\varepsilon}(x,y,t+h,\omega)$ and such that $\|\eta\|_{L^\lambda(0,t+h)}\leq \overline{C}$ where $\overline{C}=\overline{C}(\delta)$ is the constant introduced in Proposition~\ref{prp71}. 
We  define
$\xi_{\eta}(s)= \eta(\frac{t+h}{t} s)$. By Lemma~\ref{leftinv}, $\xi_{\eta}(s)$ is still horizontal with 
$$\alpha^{\xi_{\eta}}(s) = \frac{t+h}{t}\alpha^{\eta}\left(\frac{t+h}{t}s\right), \;\xi_{\eta}(0)=\eta(0)=y, \;\xi_{\eta}(t)= \eta\left(\frac{t+h}{t}t\right)=x;$$ 
so, $\xi_{\eta}(s)$ is admissible for $L^{\varepsilon}(x,y,t,\omega)$.
We observe that
\begin{eqnarray*} 
L^{\varepsilon}(x,y,t,\omega)&\leq& 
\int_0^{t }L(\delta_{1/\varepsilon}(\xi_{\eta}(s)), \alpha^{\xi_{\eta}}(s),\omega)ds\\
&=&\frac{t}{t+h}\int_0^{t+h }L\left(\delta_{1/\varepsilon}(\eta(s)), \frac{t+h}{t}\alpha^{\eta}(s),\omega\right)ds.
\end{eqnarray*} 
Define
$$
I:=\frac{t}{t+h}\int_0^{t+h }L\left(\delta_{1/\varepsilon}(\eta(s)), \frac{t+h}{t}\alpha^{\eta}(s),\omega\right)ds-
\int_0^{t+h }L(\delta_{1/\varepsilon}(\eta(s)), \alpha^{\eta}(s),\omega)ds, 
$$
so that
\begin{equation}
\label{chicago}
\int_0^{t+h }L(\delta_{1/\varepsilon}(\eta(s)), \alpha^{\eta}(s),\omega)ds-L^{\varepsilon}(x,y,t,\omega)\geq -I.
\end{equation}
Assume for the moment that
\begin{equation}\label{9sett}
I\leq C_\delta h.
\end{equation}
Then 
passing to the infimum over $\eta$ in \eqref{chicago} we obtain relation \eqref{9settembre}. \\
Let us now prove \eqref{9sett}.
Writing $\frac{t}{t+h}=1-\frac{h}{t+h}$, we have
\begin{multline}\label{9sett1}
I=\int_0^{t+h }\left(L\left(\delta_{1/\varepsilon}(\eta(s)), \frac{t+h}{t}\alpha^{\eta}(s),\omega\right)-L\left(\delta_{1/\varepsilon}(\eta(s)), \alpha^{\eta}(s),\omega\right)\right)ds\;+\\ -\frac{h}{t+h} \int_0^{t+h }L\left(\delta_{1/\varepsilon}(\eta(s)), \frac{t+h}{t}\alpha^{\eta}(s),\omega\right)ds
\end{multline}
To get \eqref{9sett} we estimate $|I|$. We start estimating the modulus of the second integral.
Note that  {\bf(L2)} implies that
\begin{equation}
\label{GrowthConditionModulus}
|L(x,q,\omega)|\leq C_1\big(|q|^\lambda+1\big)+C^{-1}_1 \big(|q|^\lambda-1\big)\leq C(|q|^\lambda+1)
\end{equation}
where we have used that $\max\{A,B\}\leq |A|+|B|$ and $C=C_1+C^{-1}_1$.\\
Then by \eqref{GrowthConditionModulus} and for $h$ sufficiently small (so that $\delta<t+h<\frac{1}{\delta}$),  we have
\begin{eqnarray}\label{412}
&&\frac{h}{t+h} \int_0^{t+h}\bigg|L\left(\delta_{1/\varepsilon}(\eta(s)), \frac{t+h}{t}\alpha^{\eta}(s),\omega\right)\bigg|ds\nn\\
&&\qquad \leq \frac{h}{t+h} \int_0^{t+h} C\left(1+\left(\frac{t+h}{t}\right)^\lambda |\alpha^\eta(s)|^\lambda \right)ds\nn\\
&&\qquad \leq \frac{h}{\delta}C
\left(\frac{1}{\delta}+\left(\frac{t+h}{t}\right)^\lambda \|\alpha^\eta\|_{L^\lambda (0,t+h)}\right)\nn\\
&&\qquad \leq h\frac{C}{\delta}
\left(\frac{1}{\delta}+\widetilde{C}\right)=C_\delta h.
\label{9sett2}
\end{eqnarray}
where the last inequality is due to Proposition~\ref{prp71}.

On the other hand, by assumptions {\bf(L2)} and {\bf(L6)}, we have
\begin{eqnarray}\label{413}
&&\int_0^{t+h}\left|L\left(\delta_{1/\varepsilon}(\eta(s)), \frac{t+h}{t}\alpha^{\eta}(s),\omega\right)-L(\delta_{1/\varepsilon}(\eta(s)), \alpha^{\eta}(s),\omega)\right|ds\nn\\
&&\qquad\leq \int_0^{t+h} C_1 \left|\left(\frac{t+h}{t}\right)^\lambda -1 \right|\left|\alpha^\eta(s)\right|^\lambda ds \nn\\
&&\qquad \leq C_1\delta^{-\lambda} \big| (|t+h)^\lambda -t^\lambda 
\big|\|\alpha^\eta\|_{L^\lambda(0,t+h)}^\lambda\nn\\
&&\qquad \leq C_{\delta} h \label{9sett3}
\end{eqnarray}
where the last inequality is due to Proposition \ref{prp71}.\\
Using \eqref{9sett2} and \eqref{9sett3} in \eqref{9sett1}, we get \eqref{9sett}.
\end{proof}


 \begin{lemma}\label{NDrem}
For later use, we collect some consequences of the previous estimates:
\begin{enumerate}
\item[1.] $L^\varepsilon(x\circ v,x,\|v\|_{CC},\omega)\le C\|v\|_{CC}$. \\
\item [2.] For $0<\rho \ll 1$ we have 
  \begin{eqnarray*}|L^\varepsilon(x,y,t,\omega)-L^\varepsilon(x,y,(1+\rho)t,\omega)|
&\le&  C\rho|L^\varepsilon(x,y,t,\omega)|,
\end{eqnarray*}      
by choosing $h=\rho t$ in the proof of 
Lemma \ref{prop1}.	
\item [3.] $L^\varepsilon(x,z,t+s,\omega)\le L^\varepsilon(x,y,t,\omega)+
L^\varepsilon(y,z,s,\omega)$.
\end{enumerate}
\end{lemma}
\begin{proof}
Point (1) follows from Lemma \ref{lemma2} by choosing $x=y$ and $t=0$. 
We prove now point (2). We note that w.l.o.g. we can replace assumption {\bf (L2)} with
\begin{equation*}
C^{-1}_1(|q|^{\lambda}+1) \leq L(x,q,\omega) \leq C_1(|q|^{\lambda}+1),
\end{equation*}
Indeed, if we increase $L$ by a constant then all the $u^\varepsilon$ increase by the same constant making no relevant change in the problem. By this assumption, Lemma~\ref{lemma71} implies $C_1^{-1}t\leq L^\varepsilon(x,y,t,\omega)$.
We now easily conclude the proof following the same arguments of Lemma \ref{prop1}: As $h(t+h)^{-1}$ becomes $\rho(1+\rho)^{-1}$ for $h=\rho t,$ the integral in \eqref{412}
can be estimated for $0 \le \rho\le 1$ by
$$
\rho \int_0^{(1+\rho)t}C_1\left((1+(1+\rho)^\lambda)|\alpha|^\lambda)\right)ds\le \rho CL^\varepsilon(x,y,(1+\rho)t,\omega)\le \hat C\rho L^\varepsilon(x,y,t,\omega), 
$$ where the first inequality follows from our lower bounds on the Lagrangian by choosing $\alpha$ as a minimizer. The second inequality is a consequence of ${(\bf L6)}$ for $s=1+\rho\le 2.$

The integral in \eqref{413} can be estimated by
$$
C_1|(1+\rho)^\lambda-1|\|\alpha^\eta\|_{L^\lambda(0,(1+\rho)t)}^\lambda.
$$For $\rho$ sufficiently small, we find a constant $C_\lambda$ such that
$
|(1+\rho)^\lambda-1|\le C_\lambda\rho.$ Now we conclude as in the previous step. 

\noindent Point (3) is obvious.
\end{proof}
As a direct consequence we have the following lemma:
\begin{lemma}\label{replacement}
Consider $x_1, y_1, x_2, y_2\in \R^N$ and $t>0$, with $\norma{-x_1\circ x_2}_{CC}+\norma{-y_1\circ y_2 }_{CC}\ll t.$
Then
\begin{multline*}
|L^\varepsilon(x_1,y_1,t,\omega)-L^\varepsilon(x_2,y_2,t,\omega)|\\
\le C\big(L^\varepsilon(x_1,y_1,t,\omega)+L^\varepsilon(x_2,y_2,t,\omega)\big)\left(\| -x_1\circ x_2 \|_{CC}+\|-y_1 \circ y_2\|_{CC}\right).
\end{multline*}
\end{lemma}
\begin{proof} 
By applying twice Lemma \ref{NDrem} part 3,  we deduce
\begin{multline*}
L^\varepsilon(x_1,y_1,t+\big(\norma{-x_1\circ x_2}_{CC}+\norma{-y_1\circ y_2}_{CC},\omega\big)\\
\le L^\varepsilon\big(x_1,x_2,\norma{-x_1 \circ x_2}_{CC},\omega\big)+L^\varepsilon(x_2,y_2,t,\omega)+L^\varepsilon(y_2,y_1,\norma{-y_1\circ  y_2 }_{CC},\omega).
\end{multline*}
By applying twice Lemma \ref{NDrem} part 1, taking first   $v=-x_1\circ x_2 $ and then  $v=-y_1\circ y_2 ,$ and by Lemma \ref{NDrem} part 2 with
$$\rho=\frac{\|-x_1 \circ x_2\|_{CC}+\|-y_1 \circ  y_2\|_{CC}}{t},$$
and interchanging the role of $x_1,x_2$ with that of $y_1,y_2$, the claim follows.
\end{proof}


\begin{theorem}\label{UniformContinuityFunctional}
Under assumptions {\bf (L2)} and  {\bf (L6)}, for all fixed  $\omega\in \Omega$, $L^{\varepsilon}(x,y,t,\omega)$ is locally uniformly continuous w.r.t.  $x,y\in \R^N$ and $t$ away from 0,  uniformly w.r.t. $\varepsilon>0$.
\end{theorem}
\begin{proof}
By Lemma \ref{prop1} we have the local uniform continuity with respect to $t$.
It remains to show the local uniform continuity with respect to $x$ (w.r.t $y$ is the same and so omitted).
We need to show that, for every $\delta>0$,  there exists a constant $\overline{C}_{\delta}>0$ such that
\begin{equation}\label{12sett}\left|L^{\varepsilon}(x,y,t,\omega)-L^{\varepsilon}(\tilde x,y,t,\omega)  \right|\leq \overline{C}_\delta\,\norma{-x\circ \tilde x}_h\qquad \forall \varepsilon>0
\end{equation}
and for any $t\in [\delta, 1/\delta]$ and for any 
$x,\tilde x, y $ with CC-norm smaller that $\frac{1}{\delta}$.\\
Indeed, we have
\begin{multline*}
\left|L^{\varepsilon}(x,y,t,\omega)-L^{\varepsilon}(\tilde x,y,t,\omega)  \right|\leq
\left|L^{\varepsilon}(x,y,t,\omega)-L^{\varepsilon}(\tilde x,y,t+\|-x\circ \tilde x\|_{h},\omega)  \right|\\+
\left|L^{\varepsilon}(\tilde x,y,t+\|-x\circ \tilde x\|_{h},\omega)-L^{\varepsilon}(\tilde x,y,t,\omega)\right|.
\end{multline*}
We observe that Lemma \ref{lemma2} (with $v=-x\circ \widetilde{x}$) and Lemma~\ref{prop1}  give
$$
\left|L^{\varepsilon}(x,y,t,\omega)-L^{\varepsilon}(\tilde x,y,t,\omega)  \right|\leq(C_1+C_\delta)
\norma{-x\circ \widetilde{x}}_h,
$$
where $C_1$ is the constant introduced in {\bf(L2)} while $C_\delta$ is the constant introduced in Lemma \ref{prop1}.
\end{proof}

\begin{lemma}\label{lemmapalle}
Under assumption {\bf(L2)}, then
\begin{equation}
 L^{\varepsilon}(x,y,t,\omega)\geq C_1^{-1}t^{1-\lambda}(d_{CC}(x,y))^{\lambda}-C_1^{-1}t, 
 \end{equation}
for all $\varepsilon>0$, $t>0$ and $x,y\in \R^N$,
where $C_1$ and $\lambda$ are the constants introduced in {\bf(L2)}.
\end{lemma}
\begin{proof}
By the definition of $L^{\varepsilon}(x,y,t,\omega)$, assumption {\bf(L2)} and Jensen's inequality, we obtain
\begin{eqnarray*}
L^{\varepsilon}(x,y,t,\omega)&\geq& C_1^{-1}\inf\limits_{\xi\in\cA^t_{y,x}} \left\{\int_0^{t }(|\alpha^{\xi}(s)|^{\lambda}-1)ds\right\}\\ 
&=&C_1^{-1}t \inf\limits_{\xi\in\cA^t_{y,x}} \left\{\left(\frac{1}{t}\int_0^{t}|\alpha^{\xi}(s)|^{\lambda}ds\right)\right\} - C_1^{-1}t\\
&\geq&C_1^{-1} t^{1-\lambda}\bigg(\inf\limits _{\xi\in\cA^t_{y,x}} \left\{\int_0^{t }|\alpha^{\xi}(s)|ds\right\}\bigg)^{\lambda}-C_1^{-1}t
\end{eqnarray*}
which is equivalent to the statement because of the definition of $d_{CC}(x,y)$.
\end{proof}
\begin{proposition}\label{prop3}
Assume that the Lagrangian~$L$ satisfies {\bf(L2)} and that the initial datum~$g$ satisfies
\begin{equation}\label{gdcc}
g(x)\geq -C(1+d_{CC}(x,0))\qquad \forall x\in \R^N.
\end{equation}
Then the $\inf\limits_{y\in \re^N}\{g(y)+L^{\varepsilon}(x,y,t,\omega)\}$ is attained in a ball.
\end{proposition}
\begin{proof}
As in  \cite[Lemma 3.4]{RT}, we want to prove that the infimum outside a suitable ball is greater than the infimum over the entire space.
From  {\bf(L2)} we have that $L(x,0,\omega)\leq C_1$, so 
$L^{\varepsilon}(x,x,t,\omega)\leq C_1t$, that implies
\begin{equation}\label{*}
\inf\limits_{y\in \re^N}\left\{g(y)+L^{\varepsilon}(x,y,t,\omega)\right\}\leq 
g(x)+L^{\varepsilon}(x,x,t,\omega)\leq g(x)+C_1t,
\end{equation}
where the second inequality is obtained choosing the constant curve $\xi(s)=x$ for any $s\in(0,t)$ in the definition of $L^{\varepsilon}(x,x,t,\omega)$.\\
From \eqref{gdcc} and the triangle inequality, we have 
\begin{align*}
g(y)+C_1^{-1}t^{1-\lambda}d_{CC}(x,y)^{\lambda}-C_1^{-1}t\geq & -C-C\,t\frac{d_{CC}(x,y)}{t}-Cd_{CC}(x,0)\\
&+C_1^{-1}t\left(\frac{d_{CC}(x,y)}{t}\right)^{\lambda}-C_1^{-1}t.
\end{align*}
Since the right-hand  side,  for any fixed $x$ and $t$, goes to $+\infty$, as $d_{CC}(x,y)\to +\infty$, then  there exist $R>0$  such that
\begin{equation}\label{**}
g(y)+C_1^{-1}t^{1-\lambda}d_{CC}(x,y)^{\lambda}-C_1^{-1}t \geq g(x)+C_1t \qquad  \forall y\in \R^N\backslash D_R
\end{equation}
where
$D_R:=\{y\in \re^N: d_{CC}(x,y)\leq RT\}.$
By using both inequalities (\ref{*}) and (\ref{**}), we get
\begin{eqnarray*}
\inf\limits_{y\in \re^N}\{g(y)+L^{\varepsilon}(x,y,t,\omega)\}&\leq&
\inf\limits_{ \re^N\backslash D_R}\left\{ g(y)+C_1^{-1}t\left(\frac{d_{CC}(x,y)}{t}\right)^{\lambda}-C_1^{-1}t\right\}\\
&\leq& \inf\limits_{\re^N\backslash D_R}\big\{ g(y)+L^{\varepsilon}(x,y,t,\omega)\big\}
\end{eqnarray*}
where the last inequality is due to Lemma \ref{lemmapalle}. To conclude we just remark that, by the H\"ormander condition, $D_R$ is contained in an Euclidean ball (up to a different radius), then the lemma is proved.
\end{proof}
\begin{proposition}\label{propu4}
Assume that $L(x,p,\omega)$ satisfies {\bf(L2)} and that $g$ is bounded. 
Then function $u^{\varepsilon}$ defined in \eqref{rapprepsilon} can be written as  
$$u^{\varepsilon}(x,t,\omega)=\min\limits_{y\in \re^N, d_{CC}(x,y)\leq C}\{g(y)+L^{\varepsilon}(x,y,t,\omega)\}$$ 
where $C>0$ is a constant depending only on the $\|g\|_{\infty}$ and the constants  {\bf (L2)}.
\end{proposition}
\begin{proof}
The boundedness of $u^{\varepsilon}$ is proved in
Theorem \ref{th1}. Moreover
for each $\eta>0$ there exists a $\overline y\in\R^N$ such that 
\begin{eqnarray*}
&&u^{\varepsilon}=\inf\limits_{y\in \re^N}\{g(y)+L^{\varepsilon}(x,y,t,\omega)\}\geq
g(\overline y)+L^{\varepsilon}(x,\overline y,t,\omega)-\eta\\
&&\geq -\|g\|_{\infty}+ C_1^{-1}t^{1-\lambda}(d_{CC}(x,\overline y))^{\lambda}-C_1^{-1}t-\eta.
\end{eqnarray*}
So we get 
\begin{equation}\label{stimau4}
(d_{CC}(x,\overline y))^{\lambda}\leq 
C_1\bigg(\|u^{\varepsilon}\|_{\infty} +\|g\|_{\infty} +C_1^{-1}t+\eta\bigg)t^{\lambda-1},
\end{equation}
which implies the statement, by recalling that $\|u^{\varepsilon}\|_{\infty}$ is bounded by $\|g\|_{\infty}$.\end{proof}

\section{A lower dimensional constrained problem to \\
determine the effective Lagrangian.}
Inspired by the approach of \cite{S1}, we now pass to study the convergence of the functional $L^{\varepsilon}(x,y,t,\omega)$ as $\varepsilon \to 0^+$, 
by using the Subadditive Ergodic Theorem.
First we introduce  a special family of horizontal curves which can be used as initial condition to build a subadditive stationary process.
At this purpose we  use curves with the property to have constant horizontal  velocity w.r.t. the given family of vector fields $\cX=\{X_1,\dots,X_m\}$, namely   $\cX$-lines.
 For more details on those curves one can see  \cite{{BD1},{BD2}}.
\begin{definition}
\label{Defi_Lines}
We  call \emph{$\cX$-line} any absolute continuous curve $\xi:[0,t]\to \R^N$, satisfying
\begin{equation}
\label{lines}
\dot{\xi}(s)=\sum_{i=1}^mq_i X_i(\xi(s))=\sigma(\xi(s))q,\quad\textrm{a.e.}\ s\in (0,t),
\end{equation}
for some constant vector $q\in \R^m$.
Using notation coherent with \cite{S1} we denote by $l^\cX_q(s)$  the  $\cX$-line starting from the origin  associated to the horizontal constant velocity $q\in \re^m$.
\end{definition}
\begin{rem}$\quad$
\begin{enumerate}
\item
Since the vector fields associated to Carnot groups are smooth, the $\cX$-lines are smooth curves so relation~\eqref{lines} holds for  all $s\in (0,t)$.
\item Since the vector fields are linearly independent at any points (see Remark \ref{RemarkLinearIndepedent}), for any fixed $q\in \R^m$ and for any fixed starting point $x$, there is a unique $\cX$-line starting from the point $x$ and associated to the horizontal constant velocity $q$.
\item $\cX$-lines starting from a given point $x$ are curves in $\R^N$ depending only on $m$ parameters with $m<N$. Then, while there always exists an horizontal curve joining two given points $x$ and $y$, in general a $\cX$-line joining $x$ to $y$ may not exist in a Carnot group.
\end{enumerate}
\end{rem}
To study the convergence of  $L^{\varepsilon}(x,y,t,\omega)$ we need to use $\cX$-lines so we first restrict our attention to the points $x$ and $y$ that can be connected by using a $\cX$-line.  Following the notations in \cite{BD2}, we define the $\cX$-plane associated to a point $x$ which is, roughly speaking the union of all the $\cX$-lines starting from  $x$. 
\begin{definition}\label{Vx}
We call $\cX$-plane associated to the point $x$ the set of all the points that one can reach from $x$ through a
$\cX$-line, i.e.
\begin{equation*}
V_x:=\{y\in\re^N\,|\,\exists\,q\in\re^m \;\textrm{and}\; \xi^{q}\; \cX\mbox{-line such that}\; \xi^{q}(0)=x, \xi^{q}(1)=y\}.
\end{equation*}
\end{definition}
\begin{rem}[$\cX$-lines in Carnot groups]
Note that in the Heisenberg group the $\cX$-lines form a subset of Euclidean straight lines but in general the structure of $\cX$-lines can look very different from the Euclidean straight lines (see \cite{{BD1},{BD2}} for some examples). Still, if  we assume that the vector fields are associated to a Carnot group in exponential coordinates (see \eqref{matrixC}), then at least the first $m$-components are Euclidean affine. This implies that in Carnot groups, whenever $y\in V_x$, then the unique horizontal velocity $q$ such that $\xi^{q}(0)=x$ and $ \xi^{q}(1)=y$ is given by $q=\pi_m(y-x)$.
\end{rem} 

Let us define 
\begin{equation}\label{defB}\cB^q_{a,b}\!\!:=\left\{\xi:[a,b]\to \R^N\,|\xi\in W^{^{1,\infty}}\!\!\big((a,b)\big)\; \textrm{horizontal,} \,\xi(a)=l^\cX_q(a),\, \xi(b)=l^\cX_q(b)\right\}.
\end{equation}
For any interval $[a,b)$ we define the following stochastic process (similar to \cite{S1}):
\begin{equation}\label{miq}
\mu_q([a,b),\omega):=
\inf_{\xi\in \cB^q_{a,b}} \;\int_a^b L (\xi(s), \alpha^{\xi}(s),\omega)ds.
\end{equation}
To use the Sub-additive Ergodic Theorem, fixed the slope $q\in \R^m$, we need to consider the action of $\Z$ on the process $\mu_q$ as additive translation in time. More precisely:
\begin{definition}
\label{tau_z} Given $a,b\in \R$ with $a<b$, $q\in \R^m$ and $z\in \Z$, we define
$$\tau_z\,\mu_q([a,b),\omega):=\mu_q(z+[a,b),\omega)
=
\inf_{\xi\in\cB^q_{a+z,b+z}}\;\int_{a+z}^{b+z} L(\xi(s), \alpha^{\xi}(s),\omega)ds.
$$
\end{definition}
\begin{lemma}\label{lemma1}
For every $q\in\re^m$, $z\in \Z$, let $\mu_q$ be defined by \eqref{miq} and $\tau_z$ the additive action introduced  in Definition \ref{tau_z}. Under assumptions {\bf(L1)-(L4)}, we have
\begin{equation}\label{eqmi}
\tau_z\,\mu_q([a,b),\omega)=\mu_q([a,b),\tau_{z_q}\omega)
\end{equation}
where $z_q={l{^\cX}_{q}(z)}$  and $l^{\cX}_q$  is the  $\cX$-line  defined in Definition \ref{Defi_Lines}.
\end{lemma}
\begin{proof}
For any $\xi\in \cB^q_{a+z,b+z}$ we consider   
\begin{equation*}
\label{scuolabus}
\widetilde{\xi}(s): =[l^\cX_q(z)]^{-1}\circ \xi (s+z).
\end{equation*}
By Lemma
\ref{leftinv} part (i),
$\widetilde{\xi}(s)$ is still horizontal and $\alpha^{\widetilde\xi}(s)=\alpha^{\xi}(s+z)$.\\
Moreover note that $\tilde \xi(a)=[l^\cX_q(z)]^{-1}\circ l^\cX_q(a+z)$ and $\tilde \xi(b)=[l^\cX_q(z)]^{-1}\circ l^\cX_q(b+z)$. \\
We claim that
\begin{equation}
\label{bumbum}
\widetilde{\xi}(a)=l^\cX_q(a)\quad  \textrm{and}\quad\widetilde{\xi}(b)=l^\cX_q(b).
\end{equation}
In fact, consider the two curves
 $l^\cX_q(s)$ and $\widetilde l^\cX_q(s):=[l^\cX_q(z)]^{-1}\circ l^\cX_q (s+z)$: both the curves start from the origin since 
$\widetilde l^\cX_q(0)=[l^\cX_q(z)]^{-1}\circ l^\cX_q (z)=0=l^\cX_q(0)$.
Moreover they have the same horizontal velocity since 
 $\alpha^{l^\cX_q}(s)=q$ (by definition) and $\alpha^{\widetilde l^\cX_q}(s)=\alpha^{l^\cX_q}(s+z)=q$ (by 
 Lemma
\ref{leftinv}). Hence by standard uniqueness for ODEs with smooth data, the two curves coincide and   in particular  \eqref{bumbum} holds.\\
This implies that   for each  $\xi\in \cB^q_{a+z,b+z}$, the curve $\widetilde{\xi}\in \cB^q_{a,b}$. Then, by the change of variable $ \widetilde s=s-z$,
\begin{align*}
\tau_z\,\mu_q([a,b),\omega)
&=
\inf_{\widetilde{\xi} \in  \cB^q_{a,b}} 
\int_{a+z}^{b+z} L([l^\cX_q(z)]\circ \widetilde{\xi} (s-z), \alpha^{\widetilde\xi}(s-z),\omega)ds\\
&=
\inf_{\widetilde\xi \in  \cB^q_{a,b}} 
\int_{a}^{b} L([l^\cX_q(z)]\circ \widetilde{\xi} (s), \alpha^{\widetilde\xi}(s),\omega)ds.
\end{align*}
We  set $z_q={l{^\cX}_{q}(z)}$  and use assumption {\bf(L4)}  to conclude
$$\tau_z\,\mu_q([a,b),\omega)=
\inf_{\xi \in \cB^q_{a,b}} \int_{a}^{b} L(\xi (s), \alpha^{\xi}(s),\tau_{z_q}\omega)ds
=\mu_q([a,b),\tau_{z_q}\omega).$$
\end{proof}
\begin{lemma}[Subadditivity] For $n\in \N$, there holds
$$
\mu_q([0,n),\omega)\le \sum_{k=1}^n\mu_q(k-1+[0,1),\omega).
$$
\end{lemma}
\begin{proof}
If $\xi_k\in \cB^q_{k-1,k},$ then we can construct a continuous horizontal curve $\xi$  in  $ \cB^q_{0,n},$ such that
$\xi(s)=\xi_k(s)$ for $s\in [k-1,k],$ $k\in \{1,\ldots,n\}.$ The claim follows  from the definition of the infimum.
\end{proof}
Under assumptions {\bf(L1)-(L4)} the 
Subadditive Ergodic Theorem applies to the process defined in
 \eqref{miq}. 
We will use it in the following form, which is taken from \cite[Prop. 1]{DalMasoModica}, based on Akcoglu and Krengel's theorem, \cite{AK}. We state it for one dimension. First,  we get the existence of a limit which may still depend on $\omega$ and then we use  ergodicity to show independence of $\omega.$ From \cite{DalMasoModica}  we recall:
\begin{defi}
We denote by ${\mathcal U}_0$ the family of all bounded measurable subsets of $\R$. For $A\in {\mathcal U}_0$, its Lebesgue measure is $|A|$.
We denote by ${\mathcal M}$ the family of subadditive functions $m:{\mathcal U}_0\to \R$ such that, for some $c>0$, there holds
$$
0\leq m(A)\leq c|A| \qquad \forall A\in {\mathcal U}_0.
$$
\end{defi}

\begin{theorem}[{\bf Subadditive Ergodic Theorem, \cite{DalMasoModica}}]
\label{DMM}
Let $\mu:\Omega\to {\mathcal M}$ be a subadditive process. If $\mu$ and $\tau_x(\mu)$ have the same law for every $x\in \mathbb Z,$ then there exists a set of full measure $\Omega'$ and a measurable function~$\phi$ such that on $\Omega'$
$$
\lim_{t\to \infty}t^{-1}|I|^{-1}\mu(\omega)(tI)=\phi(\omega)
$$ for every interval $I,$ where $|I|$ denotes its length.
\end{theorem}

We look at the points $y\in V_0\subset \R^N$. Let us recall from Definition \ref{defG} that we write $y=(y^1, y^2)\in \R^m\times\R^{N-m}$ and $y^1=\pi_m(y)$. If $y\in V_0$ then $y^2=y^2(y^1)\in\R^{N-m} $ where $y^2(\cdot)$ is a $(N-m)$-dimensional  valued function associated to the vector fields. E.g. in the case of the $n$-dimensional Heisenberg group $N=2n+1$ and $m=2n$ so  $y^2=0\in \R$, for all $y^1\in \R^{2n}$. (See \cite[Lemma 2.2]{BD2} for more details).\\
We are now ready to give a first  pointwise convergence result. 
Note that, differently from the Euclidean case, the following theorem gives the asymptotic behaviour of
$L^{\varepsilon}(0,y,t,\omega)$ only under the additional $(N-m)$-dimensional constraint expressed by  $y\in V_0$.

\begin{theorem}
\label{convzero}
Under assumptions {\bf(L1)-(L4)}, for each $t>0$ and $y\in V_0$ fixed,
\begin{enumerate}
\item The following limit exists a.s. in $\omega$ 
\begin{equation}
\label{limitvalueconstrained}
L^{\varepsilon}(y,0,t,\omega)
\longrightarrow^{\varepsilon\to 0^+}
t\;\overline{\mu}\left(\,\frac{y^1}{t}, \omega \right), 
\end{equation}
where $\overline{\mu}:\R^m\times \Omega\rightarrow \R$ is a measurable function.
\item The limit value $\overline{\mu}$ in \eqref{limitvalueconstrained} is constant in $\omega.$
\end{enumerate}
\end{theorem}
\begin{proof} {\em We first prove part 1:}
fix $q\in\R^m,$ the first step is to prove part 1 by  applying the Subadditive Ergodic Theorem \ref{DMM} to $\mu_{q}$, which tells that
\begin{equation}\label{convergenza1}
\frac{1}{\tau}\mu_{q}([0,\tau),\omega)\longrightarrow^{\tau\to +\infty} \overline{\mu}(q,\omega), \ a.s.\ \omega\in \Omega. 
\end{equation}
Note that the definition of $\mu_{q}$ involves only a one-dimensional subgroup of translations, $\{\tau_{\ell^\cX_q(z)}\}_{z\in\Z},$  the subgroup that leaves  invariant the ${\mathcal X}$-line with direction $q$, passing through the origin.

Now we  rewrite the functional $L^{\varepsilon}(y,0,t,\omega)$ defined by \eqref{Lepsilon} in terms of $\mu_{q}\big([a,b),\omega\big)$: let us prove
$$
L^{\varepsilon}(y,0,t,\omega)=\varepsilon\mu_{{\frac{y^1}{t}}}([0,\varepsilon^{-1}t),\omega).
$$
For any  $\xi \in \cA^t_{0,y}$, we define
$\widetilde{\xi}(s):=\delta_{1/\varepsilon}(\xi(s))$. 
 Using  Lemma \ref{leftinv}, part (iii) 
\begin{equation*}
L^{\varepsilon}(y,0,t,\omega)=\inf_{\widetilde{\xi}\in \cA^t_{0,y_{\varepsilon}}} \int_0^t L(\widetilde{\xi}(s), \varepsilon\alpha^{\widetilde{\xi}}(s),\omega)ds,
\end{equation*}
where $y_{\varepsilon}= \delta_{1/\varepsilon}(y)$. By the change of variable
$\widetilde{s}=s/\varepsilon$  (which we call again $s$) the previous identity becomes
\begin{equation*}
L^{\varepsilon}(y,0,t,\omega)=\varepsilon\inf_{\widetilde{\xi}\in \cA^{t/{\varepsilon}}_{0,y_{\varepsilon}}} 
\int_0^{t/\varepsilon} L(\widetilde{\xi}(\varepsilon s), \varepsilon\alpha^{\widetilde{\xi}}(\varepsilon s),\omega)d s.
\end{equation*}
Take  now $\eta( s):= \widetilde{\xi}(\varepsilon \, s)$ and note that by Lemma \ref{leftinv} $\eta(s)$ is still a horizontal curve and 
  $\alpha^{\eta}(s)=\varepsilon\,\alpha^{\widetilde{\xi}}(\varepsilon s)$, $\eta(0)=\delta_{1/\varepsilon}(y)$ and $\eta(t/\varepsilon)=0$. Then
\begin{equation}\label{Dean}
L^{\varepsilon}(y,0,t,\omega)=\varepsilon\inf_{\eta\in \cA^{t/\varepsilon}_{0,y_{\varepsilon}}} \int_0^{t/\varepsilon} L(\eta({s}),\alpha^{\eta}(s),\omega)\, d{s}.
\end{equation}
Now fix $t>0$, $y\in V_0$ and $y^1=\pi_m(y)$.
To use the convergence result in \eqref{convergenza1} it remains to show that 
$$
\cA_{0,y_{\varepsilon}}^{t/\varepsilon}=\cB^q_{a,b}\;\textrm{with}\; q=\frac{y^1}{t},\; a=0\; \textrm{and}\; b=\frac{t}{\varepsilon}.
$$
For this purpose, consider the 
$\cX$-line joining $0$ to $y$ with constant horizontal velocity $q_i=\frac{y_i}{t}$ for $i=1,\dots,m$, i.e. $l^\cX_q$ is the unique solution of 
$$\dot{l}^\cX_q(s)=\sum_{i=1}^{m}
\frac{y_i}{t} X_i(l^\cX_q(s)),\quad
l^\cX_q(0)=0,$$
(recall that $l^\cX_q(t)=y$).\\
{\em CLAIM:}  for all constant $C>0$
\begin{equation}\label{chiline}
l^\cX_q(Ct)= \delta_C\big(l^\cX_q(t)\big)= \delta_C(y).
\end{equation}
To prove claim \eqref{chiline}, let us  introduce the two curves
$l_1(s):=l^\cX_q(Cs)$ and $l_2(s):=\delta_C(l^\cX_q(s))$.
Note that  by Lemma \ref{leftinv}, parts (ii) and (iii), we have
$$l_1(0)=l_2(0)=0, \quad 
\alpha^{l_1}(s)=C\alpha^l= C\frac{y^1}{t}\quad \textrm{and}\quad
\alpha^{l_2}(s)=C\alpha^l= C\frac{y^1}{t}.$$
 This means that $l_1(\cdot)$ and $l_2(\cdot)$ both solve the ODE problem
$$\dot{x}(s)=\sum_{i=1}^{m}C\frac{y_i^1}{t} X_i(x(s)),\quad \quad x(0)=0.$$
By standard uniqueness for  ODEs with smooth data, we deduce 
$l_1(s)=l_2(s)$. This implies in particular $l_1(t)=l_2(t)$ which gives \eqref{chiline}. 
Note that here is crucial that the horizontal velocity of the two curves $l_1$ and $l_2$ is constant in time.\\

The claim \eqref{chiline} implies that $\cA_{0,y_{\varepsilon}}^{t/\varepsilon}=\cB^q_{0,t/\varepsilon}$ with $q=\frac{y^1}{t}$, thus equation \eqref{Dean} gives
\begin{equation}\label{piccione}
L^{\varepsilon}(y,0,t,\omega)= \varepsilon\mu_{{\frac{y^1}{t}}}([0,t/\varepsilon),\omega).
\end{equation}
So fixed $t>0$,  $y\in V_0$, $y^1=\pi_m(y)$ and $q=\frac{y^1}{t}$, we can rewrite  \eqref{convergenza1} as  
\begin{equation}
L^{\varepsilon}(y,0,t,\omega)=\varepsilon\mu_{\frac{y^1}{t}}([0,t/\varepsilon),\omega)\longrightarrow^{\varepsilon\to 0^+}
t\overline{\mu}\left(\frac{y^1}{t},\omega\right), \ a.s.\ \omega\in \Omega.
\end{equation}

\noindent
{\em We now prove part 2:} we show the independence of $\overline{\mu}$ from $\omega.$

Fix $z\in\R^N$ and define $l_q^z(s):=-z\circ l^{\cX}_q(s),$ then by Lemma \ref{leftinv} this is still an ${\mathcal X}$-line.
We have by  stationarity of the coefficients
$$
t\overline{\mu}\left(\frac{y^1}{t},\tau_z(\omega)\right)=\lim_{\varepsilon\to 0^+}\varepsilon\mu_{\frac{y^1}{t}}([0,t/\varepsilon),\tau_z(\omega)),
$$ 
and by  \eqref{miq} and stationarity
\begin{eqnarray*}
\mu_{\frac{y^1}{t}}([0,t/\varepsilon),\tau_z(\omega))&=&
\inf_{\xi\in \cB^q_{0,t/\varepsilon}} \;\int_0^{t/\varepsilon} L (\xi(s), \alpha^{\xi}(s),\tau_z(\omega))ds\\
 &=&\inf_{\xi\in \cB^q_{0,t/\varepsilon}} \;\int_0^{t/\varepsilon} L (-z\circ \xi(s), \alpha^{\xi}(s),\omega)ds\\
&=&\inf_{\overline{\xi}\in \cB^{q,z}_{0,t/\varepsilon}} \;\int_0^{t/\varepsilon} L (\overline{\xi}(s), \alpha^{\overline{\xi}}(s),\omega)ds
\end{eqnarray*}
where
\begin{multline*}
 \cB^{q,z}_{0,t/\varepsilon}:=\\
 \left\{\overline{\xi}:[a,b]\to \R^N\,|\overline{\xi}\in W^{^{1,+\infty}}\!\!\!\big((a,b)\big)\; \textrm{horiz,} \,\overline{\xi}(a)=-z\circ l^\cX_q(a),\, \overline{\xi}(b)=-z\circ l^\cX_q(b)\right\}.
\end{multline*}
We have to show
\begin{equation}\label{limitinv}
\lim_{\varepsilon\to 0^+}\varepsilon\inf_{\xi\in \cB^q_{0,t/\varepsilon}} \;\int_0^{t/\varepsilon} L (\xi(s), \alpha^{\xi}(s),\omega)ds
=\lim_{\varepsilon\to 0^+}\varepsilon\inf_{\overline{\xi}\in \cB^{q,z}_{0,t/\varepsilon}} \;\int_0^{t/\varepsilon} L (\overline{\xi}(s), \alpha^{\overline{\xi}}(s),\omega)ds
\end{equation}
for all $z\in \R^N,$ then $ \overline{\mu}\left(\frac{y^1}{t},\tau_z(\omega)\right)=\overline{\mu}\left(\frac{y^1}{t},\omega\right),$ so by ergodicity w.r.t to the group action 
 $\overline{\mu}\left(\frac{y^1}{t},\omega\right) $ does not depend on $\omega.$\\
We show that both infima have the same limit by connecting the endpoints $x_1:=l^\cX_q(a)$ to $y_1:=-z\circ l^\cX_q(a),$ $x_2:=l^\cX_q(b)$ to $y_2:=-z\circ l^\cX_q(b)$ by geodesics of length of order $C(q)\norma{z}_{CC}.$ Indeed, by Lemma \ref{replacement}
the difference of the cost disappears  in the limit $\varepsilon\to 0.$ (See Figure \ref{Figura-A}.)\\
 This means any path in $\cB^{q,z}_{0,t/\varepsilon}$ can be made into a path in $\cB^{q}_{0,t/\varepsilon}$ by paying a cost of order $|z|$ (for a similar argument, we refer the reader also to the proof of Lemma~\ref{LemmaLiminfnuovo}).
This extra cost vanishes in the limit after multiplication by $\varepsilon.$
\end{proof}
\begin{figure} [htbp]
\begin{center}
\includegraphics[width=3.0 in, height=5.2 in, angle=90]{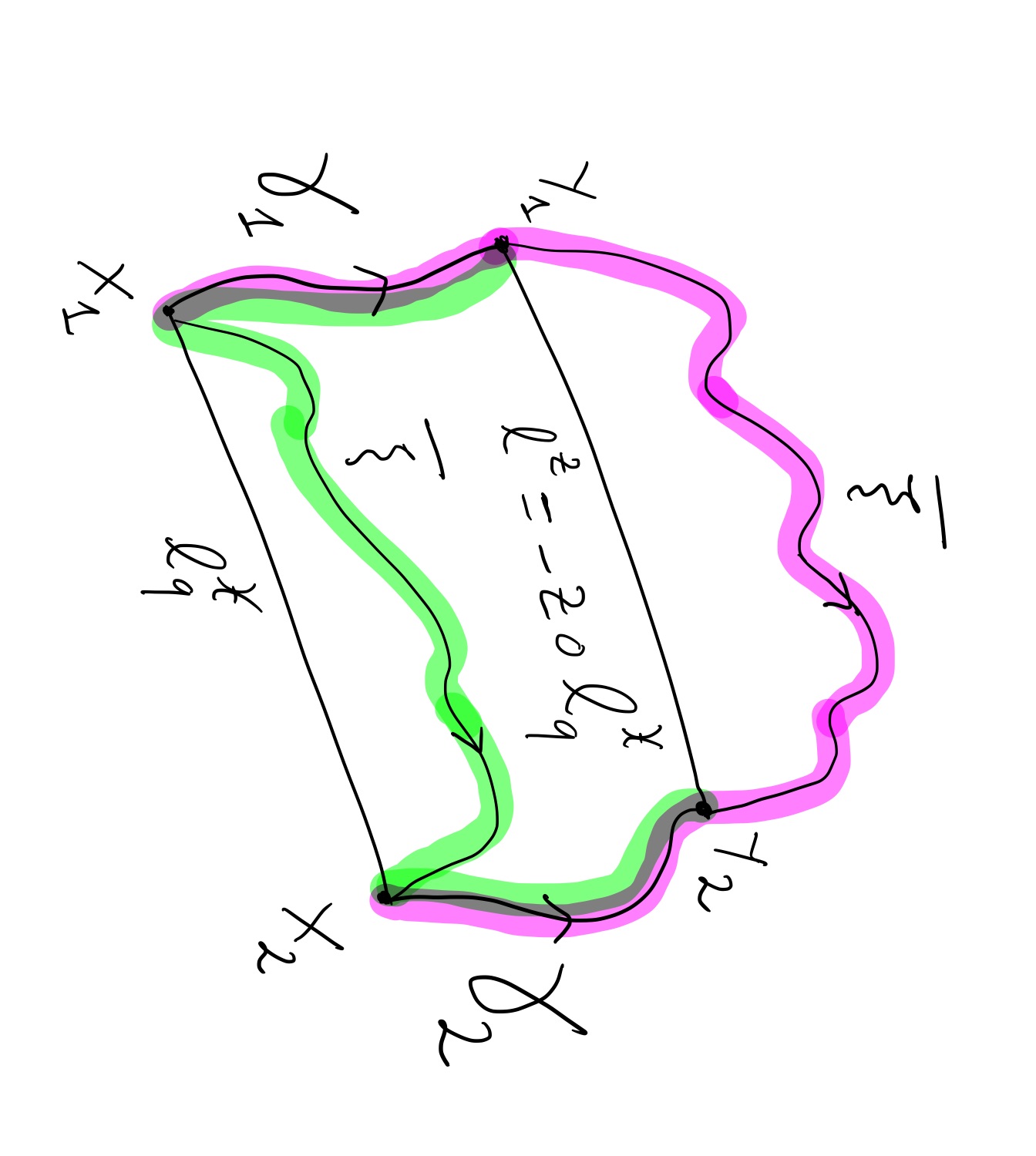}
\end{center}  
\caption{In the picture are drawn two $\cX$-lines with constant horizontal velocity $q$ connecting respectively $x_1$ with $x_2$ and $y_1=-z\circ x_1$ with $y_2=-z\circ x_2$. Then $\xi$ and $\bar{\xi}$ are respectively admissible curves touching at the two couple of points.}
  \label{Figura-A}
\end{figure}


\begin{remark}
Note that in the case of {\em short-correlated} random coefficients the proof of part 2 is  unnecessary, as the independence of $\omega$ follows already from the fact that the restriction of the random field to a $\cX$-line is again short-correlated and the zero-one-law for independent random variables.
\end{remark}
\begin{remark} Note that the convergence result \eqref{limitvalueconstrained} means  that, for all $t>0$ and $y\in V_0$ fixed,
there exists $\Omega^{t,y}
\subset \Omega$ with $\dP(\Omega^{t,y})=1$ such that
$L^{\varepsilon}(y,0,t,\omega)\rightarrow 
t\overline{\mu}\left(\frac{\pi_m(y)}{t}\right)$,  for all $\omega\in \Omega^{t,y}$. This convergence result is enough to define the effective Lagrangian but it is still too weak to obtain the convergence of the solutions of the homogenization problem.
\end{remark}
\begin{defi}[Effective Lagrangian]
We define the effective Lagrangian \\$\overline{L}:\R^m\to \R$  as
\begin{equation}
\label{EffectiveLagrangian}
\overline{L}(q):=\overline\mu(q)=\lim_{\varepsilon\to 0^+} L^{\varepsilon}\big((q,y_q^2),0,1,\omega\big),
\end{equation}
where the point $y_q^2\in \R^{N-m}$ is uniquely determined by $y^1=q\in \R^m$ for all points $(q,y_q^2)\in V_0\subset \R^N$.
\end{defi}
\begin{ex}[Heisenberg group]
In the 1-dimensional Heisenberg group there holds: 
$\overline{L}(q_1,q_{2}):=\lim_{\varepsilon\to 0^+} L^{\varepsilon}\big((q_1,q_2,0),0,1,\omega\big)$.
\end{ex}
 Using the definition of effective Lagrangian introduced in \eqref{EffectiveLagrangian} we can rewrite the limit in Theorem \ref{convzero} as follows: for all $y\in V_0$  and $t>0$ fixed, there exists a set $\Omega^{t,y}\subset \Omega$ with $\dP( \Omega^{t,y})=1$ such that
\begin{equation}
\label{limiteZero1}
\lim_{\varepsilon\rightarrow 0^+} L^{\varepsilon}(y,0,t,\omega)=t\overline{L}\left(\frac{\pi_m(y)}{t}\right),\quad \forall\; \omega\in \Omega^{t,y}.
\end{equation}
Next we want to derive the local uniform  convergence for $L^{\varepsilon}(x,y,t,\omega)$ under the constraint $y\in V_x$.  
The following proof is a simple adaptation of the ideas developed by Souganidis in \cite{S1} and later by the same author and co-authors in \cite{{Armstrong-Souga1-Carda},{Armstrong-Souga2},{Armstrong-Souga3}}. 
The main difference is that we work
 directly with the functional $ L^{\varepsilon}(x,y,t,\omega)$ and not with the solutions $u^{\varepsilon}(t,x)$.
 This will guarantee in once both the uniform convergence in $y$ (essential to pass to the infimum in the limit) and the uniform convergence in $x$ and $t$ (that will allow to apply our approximation argument in Section \ref{SecApproxX-lines}).

\begin{theorem}\label{Lsegnato}
Under assumptions {\bf(L1)-(L4)}   and {\bf (L6)} and the additional constraint $x\in V_y$, we have that
\begin{equation}
\label{UniformLimit}
\lim_{\varepsilon\rightarrow 0^+} L^{\varepsilon}(x,y,t,\omega)=
t\overline L \left(\frac{\pi_m(-y\circ x)}{t}\right)
\end{equation}
locally uniformly in $x,y,t$ and a.s. $\omega$, where $\overline{L}$ is the effective Lagrangian defined by  \eqref{EffectiveLagrangian}.
\end{theorem}
\begin{proof}
We first show that 
\begin{equation}
\label{gabbiani}
L^{\varepsilon}(x,y,t,\omega)= L^{\varepsilon}\left(-y\circ x,0,t,\tau_{\Del(y)}(\omega)\right).
\end{equation}
Note  that 
$x\in V_y\ \text {if and only if } -y\circ x \in V_0$ (recall that $y^{-1}=-y$ in exponential coordinates).
To prove \eqref{gabbiani}, for each $\xi\in \cA^t_{y,x}$ we define $\eta(s):=-y\circ \xi(s)$. By Lemma \ref{leftinv}-(i) we have
$\alpha^{\eta}(s)=\alpha^{\xi}(s)$, $\eta(0)=-y\circ x$, $\eta(t)=0$, hence
\begin{eqnarray*}
L^{\varepsilon}(x,y,t,\omega)&&
=\inf\limits_{\cA^t_{0,-y\circ x}} \int_0^t L\big(\delta_{1/\varepsilon}(y\circ\eta(s)), \alpha^{\eta}(s),\omega\big)ds\\
&&=\inf\limits_{\cA^t_{0,-y\circ x}}  \int_0^t L\big(\delta_{1/\varepsilon}(y)\circ\delta_{1/\varepsilon}(\eta(s)), \alpha^{\eta}(s),\omega\big)ds
\\
&&=\inf\limits_{\cA^t_{0,-y\circ x}}  \int_0^t L\big(\delta_{1/\varepsilon}(\eta(s)), \alpha^{\eta}(s),\tau_{\delta_{1/\varepsilon}(y)}\omega\big)ds\\
&&=L^{\varepsilon}\left(-y\circ x,0,t,\tau_{\delta_{\frac{1}{\varepsilon}}(y)}(\omega)\right),
\end{eqnarray*}
where we have used property {\bf(L4)}.\\
By combining the estimates found in Section \ref{SectionEstimates} with Egoroff's Theorem and the Ergodic Theorem we can conclude.
Since the argument is standard and has been used already in several papers, 
then we recall only some
 steps.\\
Since $L^{\varepsilon}(x,y,t,\omega)$  are equi-uniformly continuous  in $t>0$ and $x,y\in \R^N$ (see Theorem \ref{UniformContinuityFunctional}), using the density of $\Q$ in $\R$ we can restrict our attention only to points of the form $t_z\in (0,+\infty)\cap \Q=\Q^+$, $x_z\in \Q^N$ and $y_z\in \Q^N$. 
We then define the following set:
$$
\Omega_0:=\bigcap_{t_z\in \Q^+,x_z\in \Q^N,\,y_z\in \Q^N}\Omega^{t_z}_{x_z,y_z}.
$$
Note that  $\Omega_0$ does not depend anymore on $t$, $x$ and $y$ and $\dP(\Omega_0)=1$.\\
Using  the structure of Carnot group in exponential coordinates  that implies
$\pi_m(-y\circ x)=\pi_m(x)-\pi_m(y)$, since by \eqref{limiteZero1} $-y\circ x\in \Q^N$,
we know that 
$$
\lim_{\varepsilon\rightarrow 0^+} L^{\varepsilon}(-y_z\circ x_z,0,t_z,\omega)= t_z\,\overline{L}\left(\frac{\pi_m(x_z)-\pi_m(y_z)}{t_z}\right),
\quad \textrm{for all} \;\omega\in\Omega_0.
$$
Applying Egoroff Theorem, we find a ``very big'' subset of $\Omega_0$  where the convergence is uniform in $\omega$. More precisely for any
fixed $\delta>0$, there exists $A_{\delta}\subset \Omega_0$ such that 
$\dP(\Omega_0\setminus A_{\delta})\leq \delta$ (i.e. $\dP( A_{\delta})=1-\delta$) and
$$
\lim_{\varepsilon\rightarrow 0^+} L^{\varepsilon}(-y_z\circ x_z,0,t_z,\omega)= t_z\,\overline{L}\left(\frac{\pi_m(x_z)-\pi_m(y_z)}{t_z}\right),$$
uniformly for all $t_z$, $x_z$, $y_z$ and all $\omega\in A_{\delta}$.\\
To conclude one can use the  Ergodic Theorem to show that with very high probability $\tau_{\Del(y)}(\omega)\in A_{\delta}$. 
The application of the Ergodic Theorem 
 is quite technical, so we refer to Lemma 5.1. in \cite{Armstrong-Souga1-Carda} for the detailed argument. We just like to remark that by Lemma \ref{ReelationDistances} one can easily replace the Euclidean ball with the homogeneous ball (and the reverse), up to consider a different power for the radius  which depends only on the step of the Carnot group.\\
This argument  together with the  estimates in Section  \ref{SectionEstimates} (where we found a uniform modulus of continuity depending only on the assumption on $H$ and on the Carnot group) conclude the proof.
\end{proof}
\begin{corollary}
Under the assumptions of Theorem \ref{Lsegnato},  we have
\begin{equation}
\label{PartialInf}
\lim_{\varepsilon\to 0^+}\inf_{y\in V_x} \left[g(y)+L^{\varepsilon}( x,y,t,\omega)\right]=
\inf_{y\in V_x}  \left[g(y)+t \overline{L}\left(\frac{\pi_m(x)-\pi_m(y)}{t}\right)\right].
\end{equation}
\end{corollary}
\begin{proof} We just sum $g(y)$ on both sides of \eqref{UniformLimit} and take the infimum over $y\in V_x$ by using the uniform convergence in $y$.
\end{proof}
Note that the right-hand side in \eqref{PartialInf} coincides with the Hopf-Lax formula introduced in \cite[Theorem 1.1]{BCP}. Then, whenever the initial condition satisfies the additional assumption $g(x)\geq g(\pi_m(x))$ for all $x\in \R^N$, the right-hand side is the unique viscosity solution of 
the associated Hamilton-Jacobi Cauchy problem (defining $\overline{H}=\overline{L}^*$ and proving  convexity for $\overline{L}$ see Section \ref{HomogSection}) .
Unfortunately in general 
$v_{\varepsilon}(t,x)=\inf_{y\in V_x} \left[g(y)+L^{\varepsilon}( x,y,t,\omega)\right]$ 
 does not solve the $\varepsilon$-problem \eqref{ApproxPr}. 
 Then it is crucial to get rid of the additional constraint $y\in V_x$.
 For this purpose, in Section \ref{SecApproxX-lines} we will introduce a novel approximation argument, by using a suitable construction by $\cX$-lines.\\

We conclude the section investigating some properties for the effective Lagrangian that will be used later.

\begin{lemma}\label{Iq}
For any $y=(y^1,y^2)\in V_0$, we have
$$
\inf\limits_{\xi} \int_0^t |\alpha^{\xi}(s)| ds\geq |y^1|,
$$
where the infimum is taken over all the horizontal curves $\xi(s)$ such that $\xi(0)=(y^1, y^2)$ and $\xi(t)=0$.
\end{lemma}
\begin{proof} 
Given any horizontal curve $\xi$ such that $\xi(0)=(y^1,y^2)\in \R^N$, $\xi(t)=0\in \R^N$, we define $\eta:[0,t]\rightarrow \R^m$ as $\eta(s):=\pi_m(\xi(s))$.
Then $\eta(0)=y^1\in\R^m$, $\eta(t)=0\in\R^m$.
Moreover, from the structure of $\sigma$ (see \eqref{matrixC}), 
we have
$\dot{\eta}(s)= (\dot{\xi}_1(s), \dots, \dot{\xi}_m(s))=\alpha^{\xi}(s)$.\\
Then, since $\eta$ are curves in the Euclidean $\R^m$ joining $y^1$ to $0$ at time $t$,
$$
\int_0^t |\alpha^{\xi}(s)| ds\geq \inf\limits_{\eta} \int_0^t |\dot{\eta}(s)| ds=|y^1|,
$$ and we can conclude taking the infimum on the left-hand side term.
\end{proof}

\begin{proposition}\label{prop5}
$\overline L(q)$ is continuous and superlinear in $q$, i.e.
\begin{equation}\label{superlinear}
\overline L(q)\geq C_1^{-1}(|q|^{\lambda}-1)
\end{equation}
where $C_1$ and $\lambda$ are the constants introduced in {\bf(L2)}.
\end{proposition}
\begin{proof}
The continuity follows from the uniform convergence of $L^{\varepsilon}$ in \eqref{EffectiveLagrangian}.
For each $q\in\R^m$, take $y^1=q$, $y=(q, y^{2})\in V_0$ , $t=1$ and $x=0$.
$$
 L^{\varepsilon}(0,y,1,\omega)=\inf\limits_{\xi\in \cA^t_{y,0}} \int_0^1 L(\delta_{1/\varepsilon}(\xi(s)), \alpha^{\xi}(s),\omega)ds.
$$
From assumption {\bf(L2)}, Jensen's inequality and Lemma \ref{Iq}, we get
\begin{eqnarray*}
L^{\varepsilon}(0,y,1,\omega)&\geq& 
 C_1^{-1}\inf\limits_{\xi\in \cA^1_{y,0}} \left(\int_0^1 |\alpha^{\xi}(s)|^{\lambda} ds\right)-C_1^{-1}\\
& \geq & C_1^{-1} \inf\limits_{\xi\in \cA^1_{y, 0}} \left(\int_0^1 |\alpha^{\xi}(s)| ds\right)^{\lambda}-C_1^{-1}\\
& \geq & C_1^{-1}|q|^{\lambda}-C_1^{-1},
 \end{eqnarray*}
 which implies \eqref{superlinear}, passing to the limit as $\varepsilon\to 0^+$.
\end{proof}

\section{Approximation by $\cX$-lines and convergence of the variational problem.}\label{SecApproxX-lines}

To remove the constraint $y\in V_x$ the idea is to apply Theorem \ref{Lsegnato}  to suitable step-$\cX$-lines, i.e. horizontal curves whose horizontal velocity is step-constant w.r.t. the given vector fields.
More precisely we want to approximate the horizontal velocity $\alpha(t)\in \R^m$ in $L^1$ by step-constant functions.
(Recall that if two horizontal velocities are close in $L^1$-norm then the a associated horizontal curves are close in $L^{\infty}$-norm, see Lemma \ref{aprroxCurveLemma}.) 

 We will treat the liminf and the limsup separately. Both are  treated in the spirit   as one would do for the $\Gamma$-liminf and the $\Gamma$-limsup for integral functionals. One of the technical difficulties  here is how to approximate limits of a sequence of minimizing paths by $\cX$-lines. Due to the fast oscillations of our integrands in $\xi,$ this is not straightforward, we refer to the discussion in \cite{E}. As we cannot assume that our limit paths are smooth but only in some Sobolev space, we have to work with Lebesgue points of the horizontal velocity. Here this is more subtle than in the Euclidean case. 

\begin{lemma}\label{Lebesgueapprox}
Suppose $\dot \xi(s)=\alpha_1(s)X_1(\xi(s))+\ldots+\alpha_m(s)X_m(\xi(s))$ and $t_0\in \R$ is a Lebesgue point for $\alpha_1,\dots,\alpha_m$, that means
\begin{equation}\label{lebesguepointconv}
\lim_{\delta\to 0}\max_{i=1,\ldots,m}\delta^{-1}\int_{t_0-\delta}^{t_0+\delta}|\alpha_i(s)-\alpha_i(t_0)| ds=0.
\end{equation}
Consider the $\cX$-line $\ell(s):=l^{\alpha(t_0)}(s)$ i.e.
$$
\dot \ell(s)=\alpha_1(t_0)X_1(\ell(s))+\ldots+\alpha_m(t_0)X_m(\ell(s)),\quad \ell(t_0)=\xi(t_0).
$$
Then for any $\varepsilon>0$ there exists $\delta_0>0$ such that for all $\delta<\delta_0$
\begin{equation}\label{xlineapprox}
\sup_{[t_0-\delta,t_0+\delta]}d_{CC}(\xi(s),\ell(s))<\varepsilon\left(\delta+\int_{t_0-\delta}^{t_0+\delta}|\alpha(s)|\,d\,s \right).
\end{equation}
\end{lemma}

\begin{proof} We use the exponential representation of the Carnot group. Moreover, since the Carnot-Carath\'eodory distance is locally equivalent to the homogeneous distance, we estimate $\|-\ell(s)\circ \xi(s)\|_h$.

{\em Case 1: the Heisenberg group $\dH$.}
We prove in the Heisenberg group $\dH$ by explicit computations. W.l.o.g. we assume $t_0=0$; for the first two coordinates we have
$$
\xi_i(t)=\ell_i(0)+t\alpha_i(0)+t\ \underbrace{\left( 
\frac{1}{t}\int_0^t(\alpha_i(s)-\alpha_i(0))ds\right)}_{:=r_i(t)}=\ell_i(t)+t\,r_i(t).
$$
For the third coordinate we have
\begin{eqnarray*}
\ell_3(t)&=&\frac{t}{2}\left(\ell_1(0)\alpha_2(0)-\ell_2(0)\alpha_1(0)\right)+\ell_3(0)\\
\xi_3(t)&=&\int_0^t\frac{1}{2}\left(\xi_1(s)\alpha_2(s)-\xi_2(s)\alpha_1(s)\right)ds+\ell_3(0).
\end{eqnarray*}
Writing $\alpha_2(s)=\alpha_2(s)\pm \alpha_2(0)$, we get
\begin{eqnarray*}
\int_0^t\xi_1(s)\alpha_2(s)ds
&=&\frac{t^2}{2}\alpha_1(0)\alpha_2(0)+t\ell_1(0)\alpha_2(0)+\ell_1(0)t\,r_2(t)+\\&&+\int_0^t s\, r_1(s)\alpha_2(s)ds+\alpha_1(0)\int_0^ts(\alpha_2(s)-\alpha_2(0))ds,
\end{eqnarray*}
and as $\int_0^t\xi_2(s)\alpha_1(s)ds$ can be treated similarly, we have
\begin{eqnarray*}
\xi_3(t)&=&\ell_3(t)+\frac{t}{2}\left(\ell_1(0)r_2(t)-\ell_2(0)r_1(t)\right)\\
&&+\int_0^ts\big(r_1(s)\alpha_2(s)-r_2(s)\alpha_1(s)\big)ds\\&& +\int_0^ts\left[\alpha_1(0)\big(\alpha_2(s)-\alpha_2(0)\big)-\alpha_2(0)\big(\alpha_1(s)-\alpha_1(0)\big)\right]ds.
\end{eqnarray*}
Let us denote the last two lines by $R(t).$ Now
$$
(-\ell(t)\circ \xi(t))_3=\xi_3(t)-\ell_3(t)+\frac{\ell_2(t)\xi_1(t)-\ell_1(t)\xi_2(t)}{2},
$$
we get
$$
(-\ell\circ \xi)_3=R(t)+t^2\frac{\alpha_2(0)r_1(t)-\alpha_1(0)r_2(t)}{2}.
$$
All error terms can be estimated by $t^2\|\alpha\|_\infty\sup_{[0,t]}\max_{1,2}|r_i|$ but we need an estimate where the constant depends only on $\|\alpha_i\|_{L^1}$.\\
Note that $|R(t)|$ can be estimated by 
\begin{align*}
&\left|\int_0^t s\alpha_1(0)\big(\alpha_2(s)-\alpha_2(0)\big)ds\right|\le
\int_0^t t^2\frac{|\alpha_1(0)||\alpha_2(s)-\alpha_2(0)|}{t}\!\;ds \le|\alpha_1(0)||r_2(t)| t^2\\
&\left|\int_0^tsr_1(s)\big(\alpha_2(s)-\alpha_2(0)+\alpha_2(0)\big)ds\right|
\le |\alpha_2(0)|\sup_{[0,t]}|r_1|t^2+|r_2(t)|\sup_{[0,t]}|r_1|t^2,
\end{align*}
and similarly for the remaining terms.
Denoting by $r(t)$ a term vanishing with $|r_1(t)|+|r_2(t)|$, we have to show that terms of the form $\sqrt{\alpha_i(0)}\;t\,r(t)$ can be estimated in a way which can be summed over a partition of the unit interval. Since for $t\in[0,1]$
$$
\sqrt{|\alpha_i(0)|}\;t\le \frac{1}{2}|\alpha_i(0)|t+
\frac{1}{2},
$$
and
$$
|\alpha_i(0)|t\le \int_0^t|\alpha_i(s)-\alpha_i(0)|ds
+\int_0^t|\alpha_i(s)|ds =t|r_i(t)|+\int_0^t|\alpha_i(s)|ds,
$$
and the claim is shown by applying these estimates on both sub-intervals $(t_0-\delta,t_0)$ and $(t_0,t_0+\delta)$.

{\em Case 2: general case.}

{\em Step $1$.} As the left-translation leaves the CC-distance between two points invariant, we may assume w.l.o.g. $t_0=0$ and $\xi(0)=\ell(0)=0$.\\
For a general path $\eta:[0,T]\to \R^m$ we define the 1-variation norm in the following way: let $\Delta[0,T]$ be the family of all partitions of the interval $[0,T].$ Then 
$$
|\eta|_{1-var[0,T]}=\sup_{\Delta[0,T]}\sum_{(t_k,t_{k+1})\in \Delta[0,T]}|\eta(t_{k+1})-\eta(t_{k})|.
$$
It is easy to see that 
\begin{equation}\label{bus}
|\eta|_{1-var[0,T]}=\sup_{\Delta[0,T]}\sum\left|\int_{t_k}^{t_{k+1}}\dot \eta(s) ds\right|\le\int_0^t|\dot \eta(s)|ds.
\end{equation}
We define two paths in ${\mathbb R}^m$ as the projection on the first $m$ components of $\xi$ and $\ell$ and we denote them respectively by $\eta^\xi$ and $\eta^\ell$. Using the structure given by \eqref{matrixC}, we have that $\dot \eta^\xi(t)=\alpha(t)$ and $\dot \eta^\ell(t)=\alpha(0)$, then, for $\delta$ sufficiently small, \eqref{bus} implies
$$
|\eta^\xi-\eta^\ell|_{1-var[0,\delta]}\le \delta \varepsilon
$$
because by assumption $t_0=0$ is a Lebesgue point for $\alpha$.\\
{\em Step $2$.} By \cite[Proposition 7.63]{FV}, 
we have for the signature of the path (i.e. for the difference of all iterated integrals) 
\begin{align*}
&\sup_{k=1,\ldots n-m}\sup_{0<t_1<\delta}\left|\int_0^{t_1}\int_0^{t_2}\ldots\int_0^{t_k}d\eta^\xi(s_1)\otimes \ldots\otimes d\eta^\xi(s_k)ds_1 \ldots ds_k\right.\\& \left.-\int_0^{t_1}\int_0^{t_2}\ldots\int_0^{t_k}d\eta^\xi(s_1)\otimes \ldots \otimes d\eta^\xi(s_k)ds_1 \ldots ds_k\right|\delta^{-k}<C\varepsilon.
\end{align*}
 {\em Step $3$.} The Chen-Strichartz formula, which is a deep generalization of the Baker-Campbell-Hausdorff formula, allows to compute the solution of flows driven by absolutely continuous paths via multiplying the terms appearing in the signature to the corresponding commutators of the vector fields. Adapting the notation in \cite{Baudoin}, Ch. 2 to the notation used here we have for a path $\xi$ as in \eqref{EQ_Horizontal} starting from the origin (in exponential coordinates)
$$
\xi =\sum_{k=1}^r
\sum_{
I=\{i_1,...,i_k\}}
\Lambda_I(\alpha_1,\ldots,\alpha_m)X_I.
$$
Here 
$$
X_I:=[X_{i_1} , [X_{i_2} , ..., [X_{i_{k−1}} , X_{i_k} ]...].
$$
Moreover for $t_1<\ldots<t_k<\delta$
$$
\Lambda_I:=\sum_{\sigma\in S_k}
\frac{(-1)^{e(\sigma)}}{k^2
\left(\begin{array}{c}k-1\\ e(\sigma)\end{array}\right)}
\int_0^{t_1}\int_0^{t_2}\ldots\int_0^{\delta}\alpha_{i_1}(t_{i_1})
\ldots \alpha(t_{i_k})dt_{i_1} \ldots dt_{i_k},
$$ where $S_k$ is the symmetric group of $k$ elements and $e(\sigma)$ is a nonnegative integer depending only on the permutation $\sigma\in S_k$ (see \cite{Baudoin}, Ch.1 )

Note that all sums are finite as the Lie algebra is nilpotent, and that the projection of the solution onthe $k$-th layerof the graded algebra  is a multiple of a $k$-times iterated integral of the $\alpha_i,$ $i=1,\ldots,m.$ 

Combining step 2 and step 3, the desired estimate in the homogeneous (and hence Carnot-Caratheodory) distance follows.

\end{proof}


In the following lemma we build a partition using sub-intervals where we can apply the previous lemma up to a set of Lebesgue-measure arbitrarily small.

\begin{lemma}\label{covering} 
For any $\rho>0$ and $\delta>0$ there exists  natural numbers $N_1$ and $N_2,$  and a partition of $[0,1)$ formed by the union of the intervals $I_k=[t_k-\ell_k,t_k+\ell_k)$ for $k=1,\dots, {N_1}$ and the intervals $J_k=[t_k'-\ell_k',t_k'+\ell_k')$ for $k=1,\dots, {N_2}$ such that
\begin{itemize}
\item  $0<\ell_k<\rho$ for $k=1,\ldots, N_1$,
\item for $k=1,\dots,N_1$ and $i=1,\ldots,m$ we have 
$$\int_{t_k-r}^{t_k+r}|\alpha_i(s)-\alpha_i(0)|ds<r\delta,\quad \textrm{for }0<r<\ell_k,
$$
\item $\sum_{k=1}^{N_2}|J_k|<\delta.$
\end{itemize}
\end{lemma}
\begin{proof} By the Lebesgue point theorem, there exits a set ${\mathcal N}$ of zero Lebesgue measure such that any $\tau\in [0,1]\setminus{\mathcal N}$ is a joint Lebesgue point of $\alpha_1,\ldots,\alpha_m$. By definition of the Lebesgue measure, ${\mathcal N}$ can be covered by a countable union of intervals with total length smaller than $\delta$. For each $\tau\in [0,1]\setminus{\mathcal N}$ there exists a $\rho_\tau>0$ such that 
 $$
\int_{\tau-r}^{\tau+r}|\alpha_i(s)-\alpha_i(0)|ds<r\delta,\quad \textrm{for }i=1,\ldots,m \textrm{ and } 0<r<\rho_\tau.
$$
 
In this way we obtain an open cover of the compact unit intervals  and extract a finite subcover. These finitely many intervals can be ordered according to their center and made into a partition by shortening them, starting from the leftmost center, until the desired partition is obtained. 
\end{proof}

\begin{figure} [htbp]
\begin{center}
\includegraphics[scale=0.5]{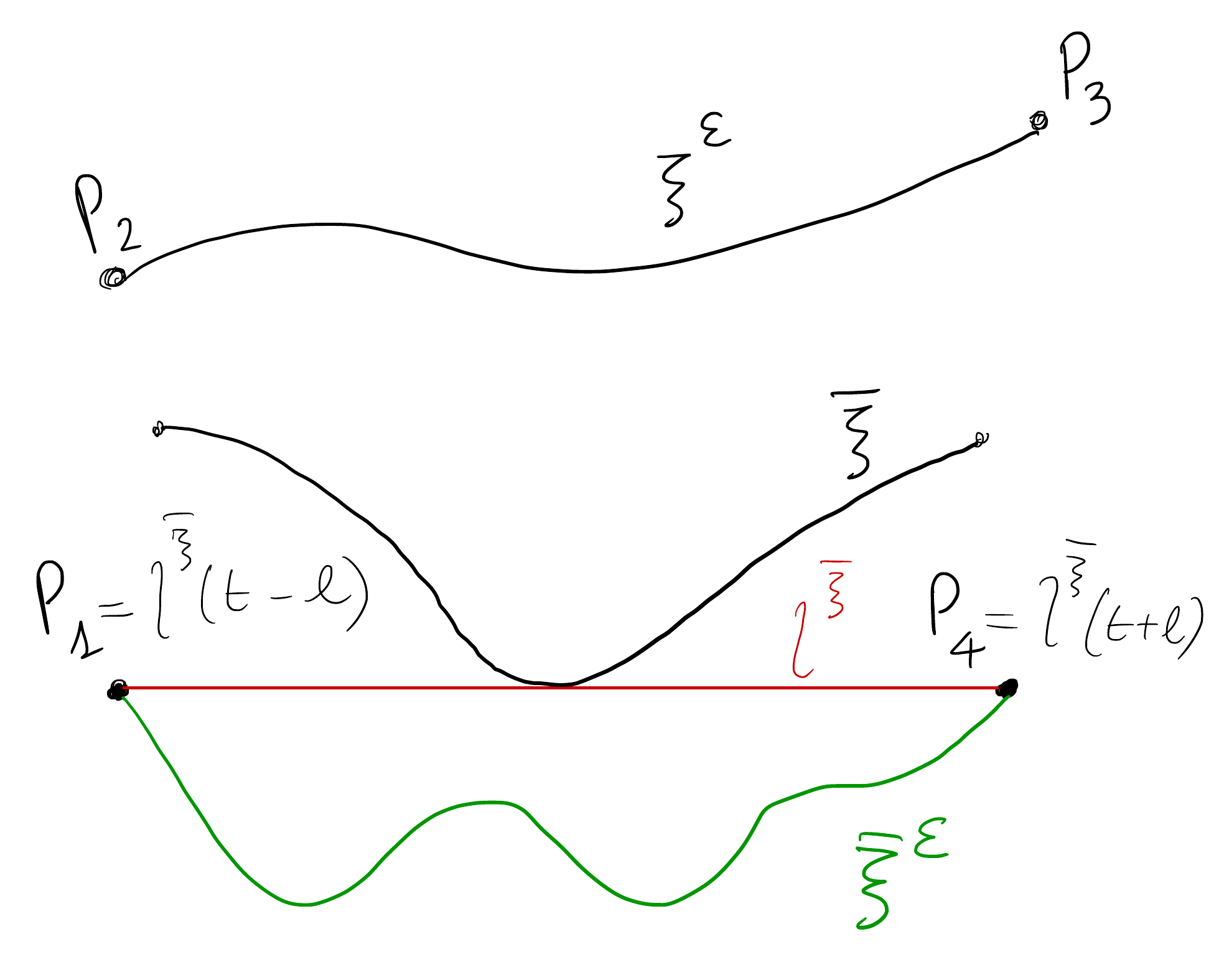}
\end{center}  
\caption{this picture illustrates Steps 3 and 4 of Lemma \ref{LemmaLiminfnuovo}.}
  \label{Figura-C}
\end{figure}


In the following lemma we prove the main lower bound for the liminf of $L^\varepsilon$.

 \begin{lemma}
\label{LemmaLiminfnuovo}
Let us assume that $L(x,q,\omega)$ satisfies assumptions {\bf(L1)-(L4)} and {\bf (L6)}.
Then, locally uniformly in $t>0$, $x,y\in \R^N$ and a.s. in $\omega\in \Omega$
\begin{equation}
\label{LIMINFN}
\liminf_{\varepsilon\to 0^+}L^{\varepsilon}(x,y,t,\omega)
\geq t\inf_{\alpha\in \cF^t_{y,x} } \int_{0}^{t} \overline{L}(\alpha(s)) ds ,
\end{equation}
  where $\cF^t_{y,x}$ is the set of all the $m$-valued measurable functions $\alpha:[0,t]\to \R^m$ such that the corresponding horizontal curve $\xi^{\alpha}(s)$ joins $y$ to $x$ in a time $t$, and $\overline L(q)$ is the effective Lagrangian defined by limit \eqref{EffectiveLagrangian}.
\end{lemma}
\begin{proof}

{\em Step 1:} For sake of simplicity we assume that $t=1$ and that for each $\varepsilon$ there exists a minimizing curve $\xi^\varepsilon$ for $L^\varepsilon(x,y,1,\omega)$. 
We observe that, by Corollary~\ref{corollary71}, the sequence $\{\xi^\varepsilon\}_\varepsilon$ is equibounded in $\mathbb L^\infty(0,1)$ and in $W^{1,\lambda}(0,1)$.
In particular, it is equi-bounded in the H\"older norm with  H\"older exponent $\gamma <1-\frac{1}{\lambda}$. Hence, by Ascoli theorem and by Sobolev embedding theorem, up to a subsequence, $\xi^\varepsilon$ uniformly converges to some H\"older curve~$\bar \xi$.
We claim that $\bar \xi$ is horizontal. Actually, by Proposition~\ref{prp71}, there holds $\|\alpha^{\xi^\e}\|_{L^\lambda}<C_2$; hence, possibly passing to a subsequence, $\{\alpha^{\xi^\e}\}$ weakly converges to some $\bar \alpha$ in $L^\lambda$. Moreover, there holds
\[
\xi^\e(t)=\sum_{i=1}^m\int_0^t \alpha^{\xi^\e}_i(s) X_i(\xi^\e(s))\, ds;
\]
so, taking into account that $\xi^\e$ uniformly converge to $\bar \xi$, we infer
\[
\bar\xi(t)=\sum_{i=1}^m\int_0^t \alpha^{\bar\xi}_i(s) X_i(\bar \xi(s))\, ds
\]
which means that $\bar \xi$ is horizontal. Finally, smoothing the horizontal velocity $\alpha^{\bar \xi}$, we obtain a family of smooth and horizontal curves uniformly approximating $\bar \xi$. Therefore, since now on, w.l.o.g. we assume that $\bar \xi$ is admissible. 

{\em Step 2:} Choose a partition of $[0,1]$ in intervals $I_k$ and $J_k$ as in Lemma \ref{covering} and error $\delta$. Denote by 
$$
{\mathcal B}:=\bigcup_{k=1}^{N_2}J_k
$$
the bad set of total length $\delta.$
By the a-priori bounds on $\|\alpha\|_{L^1}$ (see Corollary~\ref{corollary71}), assumption~{\bf (L2)} and the continuity of $\bar L$ (see Proposition~\ref{prop5}), there exists $r(\delta)\to 0$ as $\delta\to 0$ such that
\begin{eqnarray}\notag
L^{\varepsilon}(x,y,t,\omega)&=&\int_0^1 L\left(\Del(\xi^{\varepsilon}(s)),\alpha^{\xi^{\varepsilon}}(s),\omega\right) ds\\ 
\label{pizza}
&\ge &\int_{[0,1]\setminus {\mathcal B}} L\left(\Del(\xi^{\varepsilon}(s)),\alpha^{\xi^{\varepsilon}}(s),\omega\right) ds-r(\delta),
\\
\notag
\int_{0}^{1} \overline{L}(\alpha^{\overline\xi}(s)) ds&\ge& 
\int_{[0,1]\setminus {\mathcal B}}\overline{L}(\alpha^{\overline\xi}(s)) ds-r(\delta).
\end{eqnarray}
Hence we can ignore the bad intervals.\\
As the number of good intervals $N_1$ is fixed and finite and $\xi^\varepsilon$ uniformly converges to $\bar \xi$,  we can choose $\varepsilon$ sufficiently small such that 
\begin{equation}\label{convunifdCC}
\max_{k=1,\ldots N_1}\sup_{I_k}|I_k|^{-1}d_{CC}(\xi^{\varepsilon}(s), \overline\xi(s))<\delta.
\end{equation}



In the interval $I_k$, consider the constant velocity $\alpha^{\overline\xi}(t_k)$.
By the continuity of $\overline L$, the integral on the right hand side of (\ref{LIMINFN}) can be approximated by any Riemann sum, and the bad intervals can be ignored:
\begin{equation}
\label{pizza3}
\sum_{k=1}^{N_1}|I_k|\overline{L}\left(
\alpha^{\overline\xi}(t_k)
\right)\to \int_0^1\overline{L}(\alpha^{\overline\xi}(s)) ds, \quad \textrm{as}\; N_1\to +\infty. 
\end{equation}

{\em Step 3:}
Let us consider now one "good" interval, w.l.o.g. denoted by $I=[t-\ell,t+\ell],$ and consider the $\cX$-line through $\overline\xi(t)$ with velocity $\alpha^{\overline\xi}(t)$ which we denote by $l^{\overline\xi}$ (see Figure \ref{Figura-C}).


%

Let us consider a curve $\overline\xi^{\varepsilon}$ with
$\overline\xi^{\varepsilon}(t-\ell)=l^{\overline\xi}(t-\ell)$ and $\overline\xi^{\varepsilon}(t+\ell)=l^{\overline\xi}(t+\ell)$ which is the minimizer of
$$
L^\varepsilon\left(l^{\overline\xi}(t+\ell), l^{\overline\xi}(t-\ell),I,\omega\right)=\int_{t-\ell}^{t+\ell}   
L\left(\Del(\overline \xi^{\varepsilon}(s)),\alpha^{\overline \xi^{\varepsilon}}(s),\omega\right) ds.
$$
where $L^\varepsilon\left(x,y,I,\omega\right)$ is defined as in \eqref{Lepsilon} with the infimum is over the admissible curves with $\xi(t-\ell)=y$ and $\xi(t+\ell)=x$.\\ 
From Theorem~\ref{Lsegnato}, we can choose $\varepsilon$ sufficiently small such that
\begin{eqnarray}
\label{pizza2}
\int_{t-\ell}^{t+\ell} L\left(\Del(\overline\xi^{\varepsilon}(s)),\alpha^{\overline\xi^{\varepsilon}}(s),\omega\right) ds&\ge& |I| \overline{L}(\alpha^{l^{\overline\xi}})-|I|\delta,
\end{eqnarray}
This can be done uniformly for all good intervals $I_k$ as their number is already fixed.

{\em Step 4: } We now claim that 
\begin{multline}\label{claim2}
\int_{I} L\left(\Del(\xi^{\varepsilon}(s)),\alpha^{\xi^{\varepsilon}}(s),\omega\right) ds \geq\\
\int_{I} L\left(\Del(\overline\xi^{\varepsilon}(s)),\alpha^{\overline\xi^{\varepsilon}}(s),\omega\right) ds-C\left(|I|+\|\alpha^{\overline\xi}\|_{L^1(I)}\right)\delta.
\end{multline}

Let us now prove the claim \eqref{claim2}: by Lemma \ref{Lebesgueapprox} we know that
\begin{equation}\label{distanzacurve}
\sup_{I}d_{CC}(l^{\overline\xi}, \overline\xi)<C\left(|I|+\|\alpha^{\overline\xi}\|_{L^1(I)}\right)\delta.
\end{equation}
Consider the points 
$$P_1:=\overline\xi^\varepsilon(t-\ell)=l^{\overline{\xi}}(t-\ell),\quad
P_2:= \xi^{\varepsilon}(t-\ell), \quad
P_3:=\xi^{\varepsilon}(t+\ell)\quad
P_4:=\overline \xi^{\varepsilon}(t+\ell)=l^{\overline{\xi}}(t+\ell)
$$
(see Figure \ref{Figura-C}).
Then 
by \eqref{convunifdCC}, 
and \eqref{distanzacurve} 
$$
d_{CC}(P_1,P_2)\le d_{CC}(l^{\overline \xi}(t-\ell),\overline\xi(t-\ell))+d_{CC}(\overline{\xi}(t-\ell),\xi^\varepsilon(t-\ell))\le C\delta\left(
|I|+\|\alpha^{\overline\xi}\|_{L^1(I)}
\right),
$$
and analogously for $
d_{CC}(P_3,P_4)$.  By Lemmas \ref{replacement}  and \ref{lemma71} we have
$$
L^\varepsilon(P_3,P_2,I,\omega)\ge L^\varepsilon(P_4,P_1,I,\omega)-\delta C\,
\left(|I|+\|\alpha^{\overline\xi}\|_{L^1(I)}\right).
$$
Since $\xi^{\varepsilon}$ is admissible for
$
L^\varepsilon(P_3,P_2,I,\omega)
$
then
$$
L^\varepsilon(P_3,P_2,I,\omega)\le\int_{I} L\left(\Del(\xi^{\varepsilon}(s)),\alpha^{\xi^{\varepsilon}}(s),\omega\right) ds.
$$
 Combining the last two inequalities, \eqref{claim2} is shown.

{\em Step 5:}\\
Since the claim \eqref{claim2} is true for any of the good intervals $I_k$, we can easily conclude. Indeed, using respectively  that $\xi^{\varepsilon}$ are minimizer of $  L^\varepsilon(x,y,1,\omega)$, \eqref{pizza}, Definition~\ref{CC-distance_definiton} and \eqref{pizza2}
\begin{align*}
  L^\varepsilon(x,y,1,\omega)&=
  \int_0^1 L\left(\Del(\xi^{\varepsilon}(s)),\alpha^{\xi^{\varepsilon}}(s),\omega\right) ds
  \\
  &\ge \sum_{k=1}^{N_1}
\int_{I_k} L\left(\Del(\xi^{\varepsilon}(s)),\alpha^{\xi^{\varepsilon}}(s),\omega\right) ds - r(\delta)
\\ 
&\geq
\sum_{k=1}^{N_1}\int_{I_k}\!\!\!\! L\left(\Del(\overline\xi_k^{\varepsilon}(s)),\alpha^{\overline\xi_k^{\varepsilon}}(s),\omega\right) ds -\!\!\!\sum_{k=1}^{N_1}C\delta\left(|I_k|+\|\alpha^{\overline\xi}\|_{L^1(I_k)}\right)\!\!- r(\delta)\\
&\ge
\sum_{k=1}^{N_1}\int_{I_k} L\left(\Del(\overline\xi_k^{\varepsilon}(s)),\alpha^{\overline\xi_k^{\varepsilon}}(s),\omega\right) ds -C\delta(1+d_{CC}(x,y)) - r(\delta)\\
&\ge
\sum_{k=1}^{N_1}|I_k|\overline{L}\left(\alpha^{l^{\xi}}(t_k)\right)-C\delta(1+d_{CC}(x,y))-r(\delta)
\\
&\ge 
\int_0^1\overline{L}(\alpha^{\overline\xi}(s)) ds -r(\delta) -C\delta(1+d_{CC}(x,y))\\
&\longrightarrow \int_0^1\overline{L}(\alpha^{\overline\xi}(s)) ds,\quad \textrm{ as}\; \delta\to 0.
\end{align*}

\end{proof}

In the following lemma we prove the upper bound for the limsup of $L^{\varepsilon}$.

\begin{lemma}
\label{LemmaLimsup}
Let us assume that $L(x,q,\omega)$ satisfies assumptions {\bf(L1)-(L4)} and {\bf (L6)}.
Then locally uniformly in $t>0$, $x,y\in \R^N$ and a.s. in $\omega\in \Omega$
\begin{equation}
\label{LIMSUP}
\limsup_{\varepsilon\to 0^+}L^{\varepsilon}(x,y,t,\omega)
\leq t\inf_{\alpha\in \cF^t_{y,x} } \int_{0}^{t} \overline{L}(\alpha(s)) ds ,
\end{equation}
 where $\cF^t_{y,x}$ is the set of all the $m$-valued measurable functions $\alpha:[0,t]\to \R^m$ such that the corresponding horizontal curve $\xi^{\alpha}(s)$ joins $y$ to $x$ in a time $t$ and $\overline L(q)$ is the effective Lagrangian defined by limit \eqref{EffectiveLagrangian}.\end{lemma}
\begin{proof}
W.l.o.g. we  show the result for $t=1$. 
Let us choose $\overline{\alpha}\in \cF_{y,x}^1$ which realizes the infimum on the right-hand side of \eqref{LIMSUP}. We
 assume  $\overline{\alpha}$ smooth, in fact 
 we can uniformly approximate $\overline{\alpha}$ by smooth horizontal velocities and use the continuity of
 $\overline{L}$.\\ 
We fix a partition $\pi$ of the interval $[0,1]$ in $n$ equal length intervals $\left(\frac{i-1}{n},\frac{i}{n}\right)$ and we set
$$\overline{\alpha}^i=\overline{\alpha}\left(\frac{2i-1}{2n}
\right),\quad i=1,\dots,n.$$ 
We define the step-function
$$
\overline{\alpha}^{\pi}(s):=\overline{\alpha}^i=\textrm{constant},\quad \forall \, s\in \left(\frac{i-1}{n},\frac{i}{n}\right),\;\textrm{for}\; i=1,\dots,n.
$$
By 
Taylor expansion, we know that 
\begin{equation}\label{corvo}
\norma{\overline{\alpha}-\alpha^{\pi}}_{L^1(0,1)}=O\left(\frac{1}{n}\right).
\end{equation}Note that the constants in the Taylor expansion depend on higher derivatives of $\overline \alpha,$ but this is fixed throughout the proof, in particular it does not depend on $\varepsilon.$
We define the following sequence of points in $\R^N$:
$$
z^0=y , \quad z^{i}=\zeta^i(1/n),
$$
where $\zeta^i:\left[0,\frac{1}{n}\right]\to \R^N$ is the unique $\cX$-line starting from $z^{i-1}$ with constant horizontal velocity $\overline\alpha^i$ (i.e. $\dot{\zeta^i}(s)=\sum_{j=1}^m\overline\alpha_j^iX_j(\zeta^i(s))$, for $s\in\left[0,\frac{1}{n}\right]$, with $\zeta^i(0)=z^{i-1}$).
Note that $z^i\in V_{z^{i-1}}$, for all $i=1,\dots,n$.\\
\newline
\begin{figure}[htbp]
\begin{center}
\includegraphics[scale=0.4]{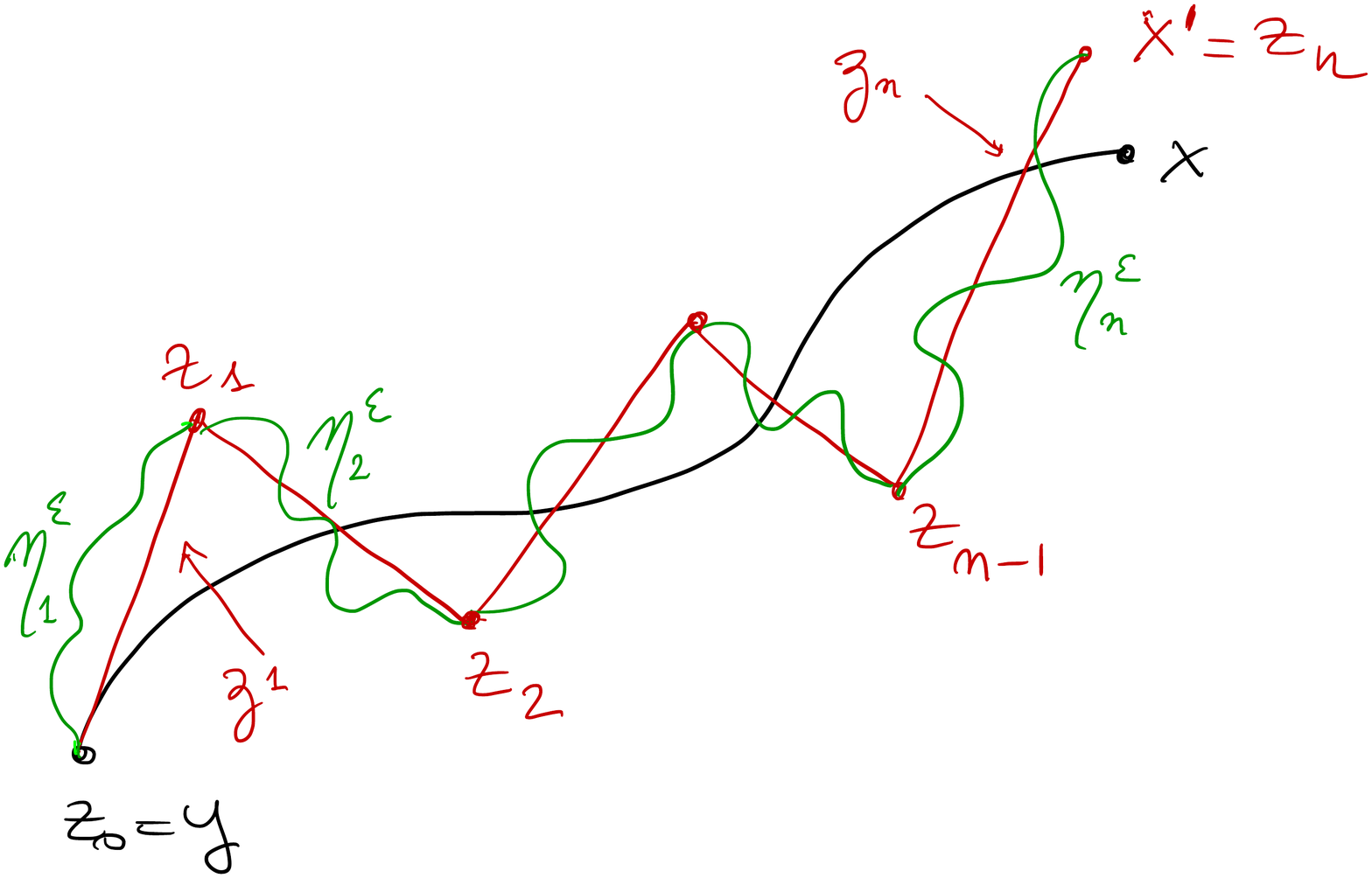}
\end{center}
\caption{this picture illustrates the arguments of the proof of Lemma \ref{LemmaLimsup}.}
\label{Figura-N}
\end{figure}
\newline
Since in general $\cX$-lines do not minimize the integral functional between the two points $z^{i-1}$ and $z^{i}$, we consider the
 curves $\eta_i^\varepsilon(s)$ which are minimisers of $L^{\varepsilon}\left(z^i ,z^{i-1},\frac{1}{n},\omega\right)$.
   We look at the two curves (see Picture \ref{Figura-N}):
\begin{align*}
& \overline{\xi}\in \cA^1_{y,x}\;\textrm{horizontal curve associated to the horizontal velocity}\;\overline{\alpha}(s),\\
 & \overline{\eta}^\varepsilon\in \cA^1_{y,x'}\;\textrm{horizontal curve defined as union of $\eta^i$, for $i=1,\dots,n$}
\end{align*}
where $x':=\zeta^n(1/n)$ satisfies, by \eqref{corvo} and  Lemma \ref{aprroxCurveLemma},
\begin{equation}\label{civetta}
|x-x'|=O\left(\frac{1}{n}\right).
\end{equation}
Note that $\overline{\eta}^\varepsilon$ is an admissible curve, i.e. $\overline{\eta}^\varepsilon\in \cA^1_{y,x'}$, but it
may be not  a minimizer for $L^{\varepsilon}\left(x',y,1,\omega\right)$. Moreover the curve $\overline{\eta}^\varepsilon$ depends on $\overline{\alpha}$, on the partition $\pi$ (i.e. on $n$) and on $\varepsilon$; nevertheless these dependences do not influence our final estimate.\\
From \eqref{civetta} and the continuity of $L^{\varepsilon}$ in $y$, uniformly
w.r.t. $\varepsilon$, denoting by $o(1)$ a function which goes to zero if $n\to +\infty$,
we get
\begin{eqnarray}
\notag
L^{\varepsilon}\left(x,y,1,\omega\right)&=&
L^{\varepsilon}\left(x',y,1,\omega\right)+o(1)\\
\notag
&\leq&
 \int_0^1L\left(\Del(\overline{\eta}^\varepsilon(s)), \alpha^{\overline{\eta}^\varepsilon}(s),\omega\right) ds+
 o(1)\nn\\
 \notag
 &=&
  \sum_{i=1}^n\int_{0}^{1/n}L\left(\Del({\eta}^\varepsilon_i(s)), \alpha^{{\eta}^\varepsilon_i}(s),\omega\right) ds+o(1)\\
   \label{speriamo}
 &=&
  \sum_{i=1}^nL^{\varepsilon}\left(z^{i},z^{i-1},\frac{1}{n},\omega\right)
+o(1)
\end{eqnarray}
where we have used the definition of $\overline{\eta}^\varepsilon$ as union of minimisers for each interval of the partition. 
 Now we first  choose $n$ big enough that the $o(1)$ term in the last line is smaller than the desired error. In the next step  we then choose $\varepsilon$ depending on $n.$\\
Let us assume for the moment the following claim
\begin{equation}\label{convetai}
L^{\varepsilon}\left(z^i, z^{i-1},\frac{1}{n},\omega\right)=\frac{1}{n}\left(\overline L(\overline\alpha^i)+r( \varepsilon)\right),
\end{equation}
where $ r( \varepsilon)$ is a function which goes to zero if  $\varepsilon\to 0$.
Hence,
by \eqref{speriamo} and \eqref{convetai}, 
we get 
\begin{eqnarray*}
L^{\varepsilon}\left(x,y,1,\omega\right)&\leq&
 \sum_{i=1}^n\frac{1}{n}\left(\overline L(\overline\alpha^i)+ r(\varepsilon)\right)+ o(1)\\
 &=&
 \sum_{i=1}^n \frac{1}{n}\overline L(\overline\alpha^i) + o\left(\frac{1}{n}\right)+ o(1)\to 
\int_{0}^{1} \overline{L}(\overline\alpha(s)) ds\quad \textrm{as } n\to +\infty.
\end{eqnarray*}
It remains only to prove \eqref{convetai}.
By  stationarity (see also \eqref{gabbiani}) we have
$$
L^{\varepsilon}\left(z^i,z^{i-1},\frac{1}{n},\omega\right)
=L^{\varepsilon}\left(-z^{i-1}\circ z^{i},0,\frac{1}{n},\tau_{\delta_{1/\varepsilon}(z^{i-1})}(\omega)\right).
$$
Using the relation between the Euclidean distance and the Carnot-Carath\'eodory distance and the fact that $\cX$-lines are horizontal curves, we get
$$
|-z^{i-1}
\circ z^i  |\leq C
d_{CC}(z^{i-1},z^{i})\leq C\int_0^{\frac{1}{n}}
|\overline{\alpha}^i| d s=C_i\frac{1}{n},$$
where $C_i=C|\overline{\alpha}^i|$.\\
Then, up to a constant, we can write
$-z^{i-1}
 \circ z^i=\frac{\overline{z}}{n}$ where $|\overline{z}|\leq 1$ (Euclidean norm in $\R^N$). Setting $\overline{z}^1=\pi_m(\overline{z})$ and using identity \eqref{gabbiani}
\begin{eqnarray*}
L^{\varepsilon}\left(z^i,z^{i-1},\frac{1}{n},\omega\right)
&=&L^{\varepsilon}\left(\frac{\overline z}{n},0,\frac{1}{n},\tau_{\delta_{1/\varepsilon}(z^{i-1})}(\omega)\right)\\
&=&\varepsilon\mu_{(-\overline z^1)}\left(\left[0,\frac{1}{\varepsilon n}\right),\tau_{\delta_{1/\varepsilon}(z^{i-1})}(\omega)\right)\\&=&
\frac{1}{n}(\varepsilon n)\mu_{\overline z^1}\left(\left[0,\frac{1}{\varepsilon n}\right),\tau_{\delta_{1/\varepsilon}(z^{i-1})}(\omega)\right)\\&=&
\frac{1}{n}\big(\overline\mu(\overline z^1)+ r(\varepsilon)\big)=
\frac{1}{n}\big(\overline L(\alpha^i)+r(\varepsilon)\big),
\end{eqnarray*}
where one can use the same argument
as in the proof of Theorem \ref{Lsegnato} to show that
$
\frac{1}{n}(\varepsilon n)\mu_{\overline z^1}([0,\frac{1}{\varepsilon n}),\tau_{\delta_{1/\varepsilon}(z^{i-1})}(\omega))\approx\frac{1}{n}(\varepsilon n)\mu_{\overline z^1}([0,\frac{1}{\varepsilon n}),\omega)
$, as $\varepsilon\to 0^+$.
\end{proof}



Combining Lemmas \ref{LemmaLiminfnuovo} and \ref{LemmaLimsup} we are finally able to prove our main convergence result.
\begin{theorem}\label{Thconvvar}
Let us assume that $L(x,q,\omega)$ satisfies assumptions {\bf(L1)-(L4)} and {\bf (L6)}.
\begin{enumerate}
\item Then
\begin{equation}\label{PensoPositivo}
\lim_{\varepsilon\to 0^+}L^{\varepsilon}(x,y,t,\omega)=t\inf_{\alpha\in \cF^t_{y,x} } \int_{0}^{t} \overline{L}(\alpha(s)) ds ,
\end{equation}
locally uniformly in $t>0$ and $x,y\in \R^N$ and almost surely $\omega\in \Omega$,
where $\cF^t_{y,x}$ is the set of all the $m$-valued measurable functions $\alpha:[0,t]\to \R^m$ such that the corresponding horizontal curve $\xi^{\alpha}(s)$ joins $y$ to $x$ in a time $t$.\\
\item
Given any  $g:\R^N\to \R$ uniformly continuous  and $u^{\varepsilon}(x,t,\omega)$ defined by \eqref{rapprepsilon}, then 
\begin{equation}
\label{LIMITFINALE}
\lim_{\varepsilon\to 0^+}u^{\varepsilon}(x,t,\omega)=
\inf_{y\in \R^N} 
\left[g(y)+
t\inf_{\alpha\in \cF^t_{y,x} } \int_{0}^{t} \overline{L}(\alpha(s)) ds 
\right],
\end{equation}
locally uniformly in $t>0$ and $x\in \R^N$ and almost surely $\omega\in \Omega$.
\end{enumerate}

\end{theorem}

\section{Homogenization for the Hamilton-Jacobi problem}\label{HomogSection}

We want to use Theorem \ref{Thconvvar} to derive the convergence of the viscosity solutions 
of problem \eqref{ApproxPr} to the unique solution of the deterministic problem \eqref{LimitProblem}.
Our strategy is to use the Hopf-Lax variational formula from \cite{BCP}.\\
The key point is the convexity of the effective Lagrangian $\overline L(q)$  defined in \eqref{EffectiveLagrangian}.
In the Euclidean case this is an easy consequence of the Dynamical Programming Principle but in our degenerate case this strategy fails since it is not possible to find three points related to a convex combination satisfying simultaneously the associated constraints.

\begin{proposition}\label{ConvexityTh}
Let us suppose that
 $L(x,q,\omega)$ satisfies assumptions {\bf(L1)-(L4)} and {\bf (L6)}.
Then $\overline L(p)$ defined in \eqref{EffectiveLagrangian} is convex in $\R^m$.
\end{proposition}
\begin{proof}
For sake of simplicity, we prove the midpoint convexity (which is equivalent to the convexity), 
i.e. we want to prove
\begin{equation}\label{TH}
\overline L\left(\frac{p+q}{2}\right)\leq \frac{1}{2}\overline L(p)+\frac{1}{2}\overline L(q)\qquad \forall p,q\in\R^m.
\end{equation}
By definition of $\overline{L}$  we have that
\begin{eqnarray}\nn
\overline L\left(\frac{p+q}{2}\right)
&=&\lim_{\varepsilon\to 0^+}L^{\varepsilon}\left(y^{\frac{p+q}{2}},0,1,\omega\right)\\
\label{defL} &=&\lim_{\varepsilon\to 0^+}\quad \inf_{_{\xi\in \cA_{0,y^{(p+q)/2}}}} \int_0^1 
L\left(\Del(\xi(s)), \alpha^{\xi}(s),\omega\right)ds
\end{eqnarray}
where $y^{\frac{p+q}{2}}= l^{(p+q)/2}(1)$ and
$l^{(p+q)/2}$ is the $\cX$-line  starting from $0$ with horizontal velocity $\frac{p+q}{2}$.
We define the curve $\xi_n$ as the horizontal curve with $\xi_n(0)=0$ and horizontal velocity
\begin{eqnarray*}
\alpha^{\xi_n}(s)= 
\left\{\begin{array}{cc}p, & if \ s\in \left[\frac{i-1}{2n}, \frac{i}{2n}\right]\; \textrm{and}\,i\ \mbox{even}, \\
q, & if \ s\in \left[\frac{i-1}{2n}, \frac{i}{2n}\right]\; \textrm{and}\, i\ \mbox{odd}
\end{array}\right.
\end{eqnarray*}
for $i=1,\dots,2n$.
We call $x_k= \xi_n(\frac{k}{2n})$, $k=0,\dots,2n$ (see Figure \ref{Figura-D}).
We observe that 
\begin{equation}\label{xi}
x_i\in\ V_{x_{i+1}}, \quad \forall \; i=1,\dots,2n.
\end{equation}
We claim that
\begin{equation}\label{vicini!}
|\xi_n(1)-y^{(p+q)/2}|=O\left(\frac{1}{n}\right),
\end{equation}
where $|\cdot|$ denotes the Euclidean norm.\\
Assume for the moment that claim \eqref{vicini!} is true. 
Then by the uniform continuity of $L^{\varepsilon}$ (see Theorem \ref{UniformContinuityFunctional}) we can deduce
\begin{equation}\label{conconv}
L^{\varepsilon}\big(y^{(p+q)/2},0,1,\omega\big)= L^{\varepsilon}\big( \xi_n(1), 0,1,\omega\big)+
O\left(\frac{1}{n}\right).
\end{equation}

We consider now the curve $\xi_{n}^{\varepsilon}$ which is the union of the curves 
$\xi_{i,n}^{\varepsilon}$ defined in $\left[\frac{i-1}{2n}, \frac{i}{2n}\right]$
that are the minimizers for 
$L^{\varepsilon}\left( x_{i+1},x_i ,\frac{1}{2n},\omega\right)$.
Observe that $\xi_{n}^{\varepsilon}$  is an admissible curve between $0$ and $\xi_n(1)$.
Hence
\begin{align}\nn
&L^{\varepsilon}(\xi_n(1), 0,1,\omega)
\leq \int_0^1 L\left(\Del(\xi_{n}^{\varepsilon}(s)), \alpha^{\xi_{n}^{\varepsilon}}(s),\omega\right)ds\\ \nn
&=\sum_{i\textrm{ odd}}\int_{\frac{i-1}{2n}}^{\frac{i}{2n}} L\left(\Del(\xi_{i,n}^{\varepsilon}(s)), \alpha^{\xi_{i,n}^{\varepsilon}}(s),\omega\right)ds\\ 
&+\sum_{i\textrm{ even}}\int_{\frac{i-1}{2n}}^{\frac{i}{2n}} L\left(\Del(\xi_{i,n}^{\varepsilon}(s)), \alpha^{\xi_{i,n}^{\varepsilon}}(s),\omega\right)ds\nn\\
&=\sum_{i\textrm{ odd}}L^{\varepsilon}\left( x_{i+1},x_i,\frac{1}{2n},\omega\right)+
\sum_{i\textrm{ even}}L^{\varepsilon}\left(x_{i+1},x_i, \frac{1}{2n},\omega\right),\label{conv1}
\end{align}
where the last identity comes from the definition of 
$\xi_{i,n}^{\varepsilon}$.

From \eqref{xi}, we can apply~\eqref{convetai} obtaining 
\begin{equation}
\label{Lconvexity2}
L^{\varepsilon}\left(x_{i+1},x_i, \frac{1}{2n},\omega\right)= 
\left\{\begin{aligned}
&\frac{1}{2n}\big(\overline L(p)+r(\varepsilon)\big),\quad \textrm{if }i\textrm{ is even} ,\\
&\frac{1}{2n}\big(\overline L(q)+r(\varepsilon)\big), \quad \textrm{if }i\textrm{ is odd}
\end{aligned}
\right.
\end{equation}
where $r(\varepsilon)\to 0$ as $\varepsilon\to 0^+$.
 By applying \eqref{defL}, \eqref{conconv},\eqref{conv1}, \eqref{Lconvexity2} we have
\begin{eqnarray*}
&&\overline L\left(\frac{p+q}{2}\right)=
\lim_{\varepsilon\to 0^+}L^{\varepsilon}\left( y^{\frac{p+q}{2}},0,1,\omega\right)=
\lim_{\varepsilon\to 0^+}L^{\varepsilon}\big(\xi_n(1), 0,1,\omega\big)+
O\left(\frac{1}{n}\right)\\ 
&&\leq\lim_{\varepsilon\to 0^+}\left(
\sum_{i\textrm{ odd}}L^{\varepsilon}\left(x_{i+1},x_i, \frac{1}{2n},\omega\right)+
\sum_{i\textrm{ even}}L^{\varepsilon}\left( x_{i+1},x_i,\frac{1}{2n},\omega\right)
\right)+O\left(\frac{1}{n}\right)\\
&&=\lim_{\varepsilon\to 0^+}\left(
\frac{1}{2}\overline L(p)+\frac{1}{2}\overline L(q)+r(\varepsilon)
\right)+
O\left(\frac{1}{n}\right).
\end{eqnarray*}
Passing to the limits, we get \eqref{TH}.

Now it  remains to prove claim \eqref{vicini!}.
First of all we estimate the distance between 
$x_2$ and $l^{(p+q)/2}(1/n)$ (recall that $x_2=\xi_n(2/2n)=\xi_n(1/n)$).
For $a=(a_1,\cdots a_m)\in\R^m$, we set $X_a:= a_1X_1+\cdots a_mX_m$.
Using the exponential coordinates,
 we  can write
$$x_2=\exp\left(\frac{1}{2n}X_p\right)\circ \exp\left(\frac{1}{2n}X_q\right)$$
and 
$$
l^{(p+q)/2}\left(\frac{1}{n}\right)= \exp\left(\frac{1}{n}X_{\frac{p+q}{2}}\right).
$$
The Baker-Campbell-Hausdorff formula (\cite{BLU}) allows to write:
\begin{eqnarray}\label{BCH}
&&x_2=\exp\left(\frac{1}{2n}X_p\right)\circ \exp\left(\frac{1}{2n}X_q\right)=\sum_{k_1,k_2}\frac{1}{k_1!k_2!}\left(\frac{1}{2n}X_p\right)^{k_1}\left(\frac{1}{2n}X_q\right)^{k_2}.
\end{eqnarray}
Moreover
\begin{eqnarray}\label{serie}
l^{(p+q)/2}\left(\frac{1}{n}\right)= \exp\left(\frac{1}{n}X_{\frac{p+q}{2}}\right)=\sum_{k}\frac{1}{k!}\left(\frac{1}{2n}X_{p+q}\right)^{k}.
\end{eqnarray}
Hence considering the first three terms of expansions \eqref{BCH} and \eqref{serie}
we obtain that
$\big|x_2-l^{(p+q)/2}(1/n)\big|= O(1/n^2)$. Iteratively, we get $\big|x_{2i}-l^{(p+q)/2}(i/n)\big|= O(1/n^2)+iO(1/n^2)$; in particular, since $x_{2n}=\xi(1)$ and $y^{\frac{p+q}{2}}= l^{(p+q)/2}(1)$, for $i=n$ we obtain the claim \eqref{vicini!}.
\end{proof}
\begin{figure} [htbp]
\begin{center}
\includegraphics[scale=0.4]{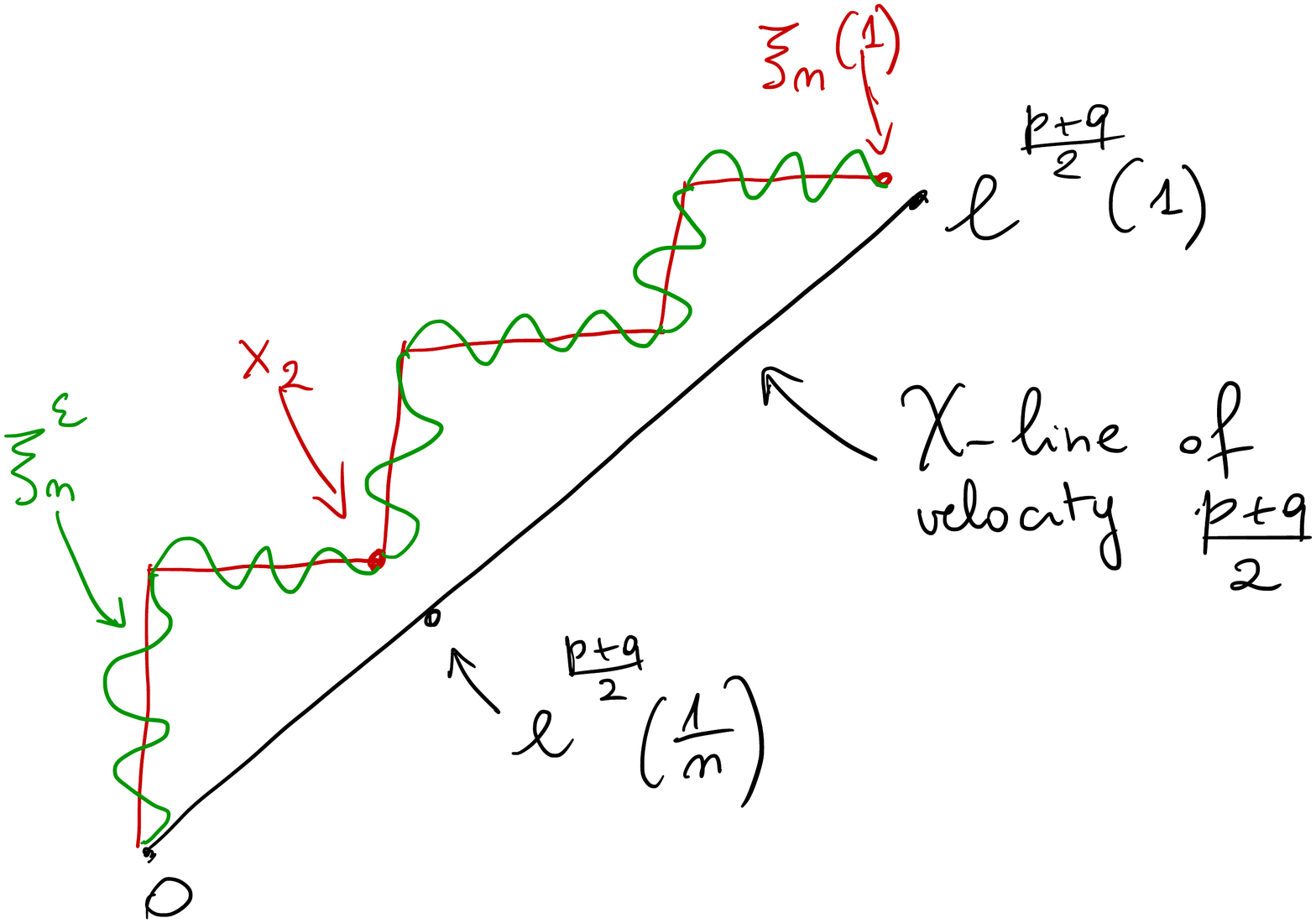}
\end{center}  
\caption{this picture illustrates the arguments of the proof of Proposition \ref{ConvexityTh}.}
  \label{Figura-D}
\end{figure}

%
%
%

We can now prove Theorem \ref{mainTH}, i.e. the homogenization result  for non-coercive  Hamilton-Jacobi problem.
\begin{proof}[Theorem \ref{mainTH}]
Note that by Lemma \ref{PropertiesLagrangian}, assumptions {\bf (H1)-(H4)} implies {\bf (L1)-(L4)}.
By Theorem~\ref{th1}, the function $u^{\varepsilon}(t,x,\omega)$ in the left-hand side of 
\eqref{LIMITFINALE} is the unique viscosity solution of \eqref{ApproxPr}.
We denote by $\bar u$ the right-hand side of \eqref{LIMITFINALE}, i.e.
\begin{equation*}
\bar u(t,x):=\inf_{y\in \R^N} \left[g(y)+ t\inf_{\alpha\in \cF^t_{y,x} } \int_{0}^{t} \overline{L}(\alpha(s)) ds 
\right],
\end{equation*}
We define the effective Hamiltonian $\overline{H}(q): =\overline{L}^*(q)$.
The convexity  and the superlinearity of $\bar L$ (see Theorem \ref{ConvexityTh} and Proposition~\ref{prop5}) imply $\overline{L}(q) =(\overline{L}^*)^{^*}(q)=\overline{H}^*(q)$.
By the Hopf-Lax formula in \cite[Theorem 3.4]{BCP}, the function~ $\bar u (t,x)$ is the unique viscosity solution of \eqref{LimitProblem}.
The convergence easily follows from \eqref{LIMITFINALE}.
\end{proof}
%
   \section{Appendix}
   In this appendix we prove Theorem~\ref{th1} on the well posedness of problem \eqref{ApproxPr}.
   
\begin{proposition}\label{UC}
Under the assumptions of Theorem \ref{th1}, let $u^\varepsilon$ be the function defined in \eqref{rapprepsilon}.
Then $u^{\varepsilon}$ is uniformly continuous in $[0,T]\times \R^N$.
\end{proposition}
\begin{proof}
We want to prove that for any $\overline \eta>0$ there exists $\overline \delta>0$ such that 
\begin{equation}\label{UCue}
|u^{\varepsilon}(t,x,\omega)-u^{\varepsilon}(s,y,\omega)|<\overline \eta,\qquad\textrm{if } |t-s|+\|-y\circ x\|_{CC}<\overline\delta.
\end{equation}
{\em Step 1.} We claim that for any $\eta >0$ there exists  $\delta>0$ such that 
\begin{equation}
\label{Step1_Appendix}
|u^{\varepsilon}(t,x,\omega)-u^{\varepsilon}(0,y,\omega)|<\eta + m(d_{CC}(x,y)),\ \forall t\in[0,\delta], \forall x,y\in\R^N.
\end{equation}
Indeed, by definition of $u^{\varepsilon}$  and {\bf (L2)} we have
\begin{equation*}
u^{\varepsilon}(t,x,\omega)-u^{\varepsilon}(0,x,\omega)\leq L^{\varepsilon}(x,x,t,\omega)\leq \int_0^t L(\Del(x),0,\omega)ds\leq C_1t.
\end{equation*}
On the other hand, by the assumption on $g$, for any $\overline \eta>0$ there exists $\overline y\in\R^N$ such that 
\begin{equation*}
u^{\varepsilon}(t,x,\omega)-u^{\varepsilon}(0,x,\omega)
\geq -m(d_{CC}(\overline y,x))+L^{\varepsilon}(x,\overline y,t,\omega)-\overline\eta.
\end{equation*}
Moreover, by {\bf (L2)}, for any $\overline\eta$ there exists $\xi\in \cA^t_{\overline y,x}$ such that
\begin{eqnarray*}
&&L^{\varepsilon}(x,\overline y,t,\omega)
\geq \int_0^t C_1^{-1}\left(|\alpha^{\xi}(s)|^{\lambda}-1\right)ds-\overline\eta \geq -C_1^{-1}t-\overline\eta.
\end{eqnarray*}
By the last two inequalities and by \eqref{stimau4}, we get 
\begin{eqnarray*}
&&u^{\varepsilon}(t,x,\omega)-u^{\varepsilon}(0,x,\omega)\geq\\ 
&&-m\bigg(C_1^{1/\lambda}(\|u^{\varepsilon}\|_{\infty}+\|g\|_{\infty}+C_1^{-1}t+1)^{1/\lambda}t^{1/\lambda-1}\bigg)-C_1^{-1}t-2\overline\eta,
\end{eqnarray*}
where the bound of $\|u^\varepsilon\|_\infty$ is proved in Theorem~\ref{th1}.
Hence, for $t$ sufficiently small and by the assumption on $g$, we get \eqref{Step1_Appendix}.

{\em Step 2.} We claim that, for any $\delta_1 >0$ there exists  $m_{\delta_1}$ such that 
\begin{equation}\label{Step2_Appendix}
|u^{\varepsilon}(t,x,\omega)-u^{\varepsilon}(s,x,\omega)|<m_{\delta_1}(|t-s|),\ \forall t,s \geq \delta_1, \forall x\in\R^N.
\end{equation}
Indeed, for any $\eta>0$ there exists $\overline y\in\R^N$ such that 
\begin{eqnarray*}
&&u^{\varepsilon}(t,x,\omega)-u^{\varepsilon}(s,x,\omega)
\geq L^{\varepsilon}(x,\overline y,t,\omega)-L^{\varepsilon}(x,\overline y,s,\omega)-\eta.
\end{eqnarray*}
The other inequality is similar. Using Lemma \ref{prop1}, we get \eqref{Step2_Appendix}.

{\em Step 3.} We claim that for any $\delta_2 >0$ there exists  $m_{\delta_2}$ such that 
\begin{equation}\label{Step3_Appendix}
|u^{\varepsilon}(t,x,\omega)-u^{\varepsilon}(t,y,\omega)|<m_{\delta_2}(\|y^{-1}\circ x\|_{CC}),\ \forall t\geq \delta_2, \forall x,y\in\R^N.
\end{equation}
Indeed, arguing as in Step~2, it is enough to prove 
$$|L^{\varepsilon}(x,z,t,\omega)-L^{\varepsilon}(y,z,t,\omega)|\leq m_{\delta_2}(\|y^{-1}\circ x\|_{CC}).$$
Actually, adding and subtracting $L^{\varepsilon}(y,z,t+\|-y\circ x\|_{CC},\omega))$, and using Lemma~\ref{lemma41}, we can conclude \eqref{Step3_Appendix}.


{\em Step 4} 
W.l.o.g. assume $s\geq t$ and $\delta$ sufficiently small. For $0\leq t\leq s\leq \delta$, from step 1, we have
\begin{equation*}
|u^{\varepsilon}(t,x,\omega)-u^{\varepsilon}(s,y,\omega)|
\leq 2\eta+m_{\delta}(d_{CC}(x,y)).
\end{equation*}
For $s>\delta$ and $|t-s|<\bar \delta$ (with $\bar \delta<\delta/2$), by step~2 and step~3, we get 
$$|u^{\varepsilon}(t,x,\omega)-u^{\varepsilon}(s,y,\omega)|\leq m_{\bar\delta}(\|-y\circ x\|_{CC}+|t-s|).$$
To conclude it suffices to choose $\bar \delta$ sufficiently small.
\end{proof}

Here we state the following result that will play a crucial role for the proof of Theorem \ref{th1} .
\begin{lemma}
\label{Step2}
For every $R, T>0$ there exists $\mu=\mu(T,R)>0$ such that for every $(t,x)\in (0,T)\times B_R(0)$ there holds
$$
u^{\varepsilon}(t,x,\omega)=\inf\big\{g(\xi(0))+\int_0^t H^*\left(\Del(\xi(s)), \alpha^{\xi}(s),\omega\right)ds\big\}$$
where the infimum is over all the $\alpha\in \cF^t_x$ with $|\alpha^{\xi}|\leq \mu(R,T)$ 
where $ \cF^t_x$ is the set of
all the $m$-valued measurable functions $\alpha$ such that the corresponding horizontal curve is $\xi^{\alpha}(s)$ with $\xi^{\alpha}(t)=x$.
\end{lemma}
\begin{proof}
The proof is the same as \cite[Theorem 2.1]{BCP}  using  \cite[Theorem 7.4.6]{CS}. 
\end{proof}

Let us now prove Theorem \ref{th1}.

\begin{proof}[Theorem \ref{th1}]


We only sketch the proof;  for the detailed calculations we refer the reader to \cite[Section 10.3.3]{Evans} and to \cite{BCP}.
First of all we prove that $u^{\varepsilon}$ is a solution of \eqref{ApproxPr}.
%
We observe that by Lemma \ref{Step2} $u^\varepsilon$ satisfies the following optimality condition: for any $0\leq h\leq t$ we have 
$$
u^{\varepsilon}(t,x,\omega)=\inf\left\{\int_{t-h}^t L\left(\Del(\xi(s)), \alpha^{\xi}(s),\omega\right)ds+ u^{\varepsilon}(t-h,\xi(t-h),\omega)\right\},$$
where the infimum is over all the $\alpha\in \cF^t_x$ with $|\alpha^{\xi}|\leq \mu(R,T)$.\\

From assumption {\bf(L2)}, Proposition \ref{prp71} and \cite[Lemma 10.3.3]{Evans}, we get
\begin{equation*}\label{bnd}
\|u^{\varepsilon}\|_{\infty}\leq C,\qquad
\textrm{for any compact }K\subset \R^N,\quad 
\|u^{\varepsilon}\|_{W^{1,\infty}([0,T]\times K)}\leq C_K.
\end{equation*}
Following the same arguments of \cite[Theorem 2, Section 10.3.3]{Evans} and \cite{BCP}, we get that $u^\varepsilon$ fulfills $u^{\varepsilon}(0,x)=g(x)$ and it is a viscosity solution of
$$u_t+\mathcal H(x, Du)=0,$$
where 
$\mathcal H(x, Du)=\max_{a\in\R^m, |a|\leq \mu(R,T)}\{p\cdot \sigma(x)a-L(x,a)\}$.\\
Arguing as in \cite[equation (45) and proof of Theorem 3.2]{BCP} we get that, if $u$ is differentiable, then $\mathcal H(x, Du)=H(x,\sigma Du)$.
Applying this property to a smooth test function, we conclude that
 $u^{\varepsilon}$ is a viscosity solution of problem \eqref{ApproxPr}.

The uniqueness of the solution follows from the uniform continuity of $u^\varepsilon$ (see Proposition~\ref{UC}) and applying the result of Biroli \cite[Theorem 4.4]{Bir}.
 \end{proof}
 
Acknowledgments: 
The first author was partially supported by the Leverhulme Trust via grant RPG-2013-261 and by EPSRC via grant EP/M028607/1.
The second author was partially supported by the EPSRC Grant ``Random Perturbations of ultra-parabolic PDEs under rescaling''.
The third and forth authors are members of GNAMPA-INdAM and were partially supported also by the research project of the University of Padova "Mean-Field Games and Nonlinear PDEs" and by the Fondazione CaRiPaRo Project "Nonlinear Partial Differential Equations: Asymptotic Problems and Mean-Field Games".\\
The authors would also like to thank  P.E. Souganidis and R. Monti for the many interesting conversations.


\vskip 0.3truecm


\end{document}